\documentclass[10pt,twoside,a4paper,english,leqno]{article}
\title{The multilayer shallow water system in the limit of small density contrast}
\author{Vincent Duch\^ene%
\thanks{IRMAR - UMR6625, CNRS and Universit{\'e} de Rennes 1 ({\tt vincent.duchene@univ-rennes1.fr}).}}
\date{\today}

\usepackage{babel} 
\usepackage[utf8]{inputenc} 
\usepackage[T1]{fontenc}
\usepackage[ec]{aeguill}
\usepackage[thinlines]{easymat}  
\usepackage{amsmath, amsthm} 
\usepackage{amsfonts,amssymb}
\usepackage{float} 
\usepackage[pdftex]{graphicx}
\usepackage{fancyhdr} 
\usepackage{url} 
\usepackage{hyperref} 
\usepackage{mathtools} 
\usepackage{bm}

\numberwithin{equation}{section}

\makeatletter
\let\Title\@title
\let\Author\@author
\makeatother
 
\fancypagestyle{mystyle}{%
	\fancyfoot{}
	\fancyhead[CO]{\Title} 
	\fancyhead[CE]{\scriptsize Vincent Duch\^ene} 
	\fancyhead[RE]{\scriptsize\today}  
	\fancyhead[LO]{}  
	\fancyhead[LE,RO]{\scriptsize\thepage}
	}
\fancypagestyle{plain}{\fancyhead{} } 

\newcommand{\RR}{\mathbb{R}}
\newcommand{\CC}{\mathbb{C}}
\newcommand{\ZZ}{\mathbb{Z}_{h_0}}
\newcommand{\NN}{\mathbb{N}}
\newcommand{\VV}{\mathbb{V}_{h_0}}

\newcommand{\e}{\textbf{e}}
\newcommand{\m}{\mathfrak{m}}
\newcommand{\mT}{\mathcal{T}}
\newcommand{\mF}{\mathcal{F}}
\newcommand{\mS}{\mathcal{S}}
\newcommand{\mR}{\mathcal{R}}

\newcommand{\M}{\mathcal{M}}
\renewcommand{\O}{\mathcal{O}}

\renewcommand{\u}{\mathbf{u}}
\renewcommand{\v}{\mathbf{v}}
\newcommand{\w}{\mathbf{w}}
\newcommand{\x}{\mathbf{x}}

\newcommand{\z}{\mathbf{0}}

\renewcommand{\L}{{\sf L}}
\newcommand{\C}{{\sf C}}
\newcommand{\D}{{\sf D}}
\newcommand{\U}{{\sf U}}
\newcommand{\V}{{\sf V}}
\newcommand{\W}{{\sf W}}
\newcommand{\A}{{\sf A}}
\newcommand{\B}{{\sf B}}
\newcommand{\J}{{\sf J}}
\newcommand{\Z}{{\sf Z}}
\renewcommand{\S}{{\sf S}}
\newcommand{\R}{{\sf R}}
\renewcommand{\P}{{\sf P}}
\newcommand{\Q}{{\sf Q}}
\newcommand{\T}{{\sf T}}

\newcommand{\eqdef}{\stackrel{\rm def}{=}}
\renewcommand{\t}{\tilde}
\newcommand{\dd}{{\rm d}}
\DeclareMathOperator{\diag}{diag}
\DeclareMathOperator{\vect}{span}
\DeclareMathOperator{\ran}{ran}
\DeclareMathOperator{\tr}{tr}
\DeclareMathOperator{\Ro}{Ro}
\newcommand{\Id}{{\sf Id}}

\renewcommand{\r}{\varrho}

\newcommand{\RL}{{\rm RL}}

\newcommand{\init}{{\rm in}}

\newcommand{\into}[1]{[0,#1)}
\newcommand{\intf}[1]{[0,#1]}

\newcommand{\Pif}{{\sf \Pi}_{\rm f}}
\newcommand{\Pis}{(\Id-{\sf \Pi}_{\rm f})}
\newcommand{\Pifx}{{\sf \Pi}_{\rm f}^x}
\newcommand{\Pisx}{(\Id-{\sf \Pi}_{\rm f}^x)}

\newcommand{\dsp}{\displaystyle}

\newcommand{\ie}{{\em i.e.}~}
\newcommand{\eg}{{\em e.g.}~}
\newcommand{\id}[1]{\left\vert_{_{#1}}\right.}
\DeclarePairedDelimiter\abs{\lvert}{\rvert}

\DeclarePairedDelimiter\norm{\big\lvert}{\big\rvert}
\DeclarePairedDelimiter\Norm{\big\lVert}{\big\rVert}
\DeclarePairedDelimiter\bNorm{\big\lvert\hspace{-1pt}\big\lvert\hspace{-1pt}\big\lvert}{\big\lvert\hspace{-1pt}\big\lvert\hspace{-1pt}\big\lvert}

\newtheorem{Theorem}{Theorem}[section]
\newtheorem{Definition}[Theorem]{Definition}
\newtheorem{Proposition}[Theorem]{Proposition}
\newtheorem{Corollary}[Theorem]{Corollary}
\newtheorem{Lemma}[Theorem]{Lemma}
\newtheorem{Remark}[Theorem]{Remark}

\begin{document}
\maketitle

\begin{abstract}
We study the inviscid multilayer Saint-Venant (or shallow-water) system in the limit of small density contrast. We show that, under reasonable hyperbolicity conditions on the flow and a smallness assumption on the initial surface deformation, the system is well-posed on a large time interval, despite the singular limit. By studying the asymptotic limit, we provide a rigorous justification of the widely used rigid-lid and Boussinesq approximations for multilayered shallow water flows. The asymptotic behaviour is similar to that of the incompressible limit for Euler equations, in the sense that there exists a small initial layer in time for ill-prepared initial data, accounting for rapidly propagating ``acoustic'' waves (here, the so-called barotropic mode) which interact only weakly with the ``incompressible'' component (here, baroclinic). 
\end{abstract}

\tableofcontents

\section{Introduction}
\subsection{Presentation of the models and the problem}
This work dedicated to the study of the so-called multilayer Saint-Venant system, which arises as an approximate model for the propagation of waves in the ocean or atmosphere, when density stratification cannot be neglected. We will refer to as {\em free surface system} the following  first-order, quasilinear system of $N(1+d)$ coupled evolution equations:
\begin{equation}\label{FS-U-intro}
\left\{ \begin{array}{l}
\displaystyle\partial_{ t}{\zeta_n} \ + \ \sum_{i=n}^N\nabla\cdot (h_i\u_i) \ =\ 0,  \\ 
\\
\displaystyle\partial_{ t}\u_n \ + \ \frac{g}{\rho_n}\sum_{i=1}^n (\rho_i-\rho_{i-1})\nabla\zeta_i\ + \  (\u_{n}\cdot\nabla)\u_n  \ = \ \z\ .
\end{array}
\right. \qquad (n=1,\dots,N)
\end{equation}
Here, the unknowns $\zeta_n(t,\x)$ and $\u_n(t,\x)$ represent respectively the deformation of the $n^{\rm th}$ interface and the layer-averaged horizontal velocity in the $n^{\rm th}$ layer, at time $t$ and horizontal position $\x\in\RR^d$ where $d\in\{1,2\}$; see Figure~\ref{F.SketchOfTheDomain}. If $d=2$, then we denote $\x=(x,y)$ and $\u_n=(u_n^x,u_n^y)$. We denote by $\rho_n>0$ the mass density of the homogeneous fluid in the $n^{\rm th}$ layer, whereas $g$ is the gravitational acceleration. Finally, 
\[ h_n(t,\x)\eqdef \delta_n+\zeta_n(t,\x)-\zeta_{n+1}(t,\x)>0\]
 is the depth of the $n^{\rm th}$ layer. By convention, we set $\rho_0=0$ (above the upper free surface is vacuum), and $\zeta_{N+1}(t,\x)\equiv 0$ (the bottom is flat). We restrict ourselves to the setting of stable stratification, namely
\[ 0<\rho_1<\dots<\rho_N.\]

\begin{figure}[htbp]
\begin{center}\includegraphics[width=.7\textwidth]{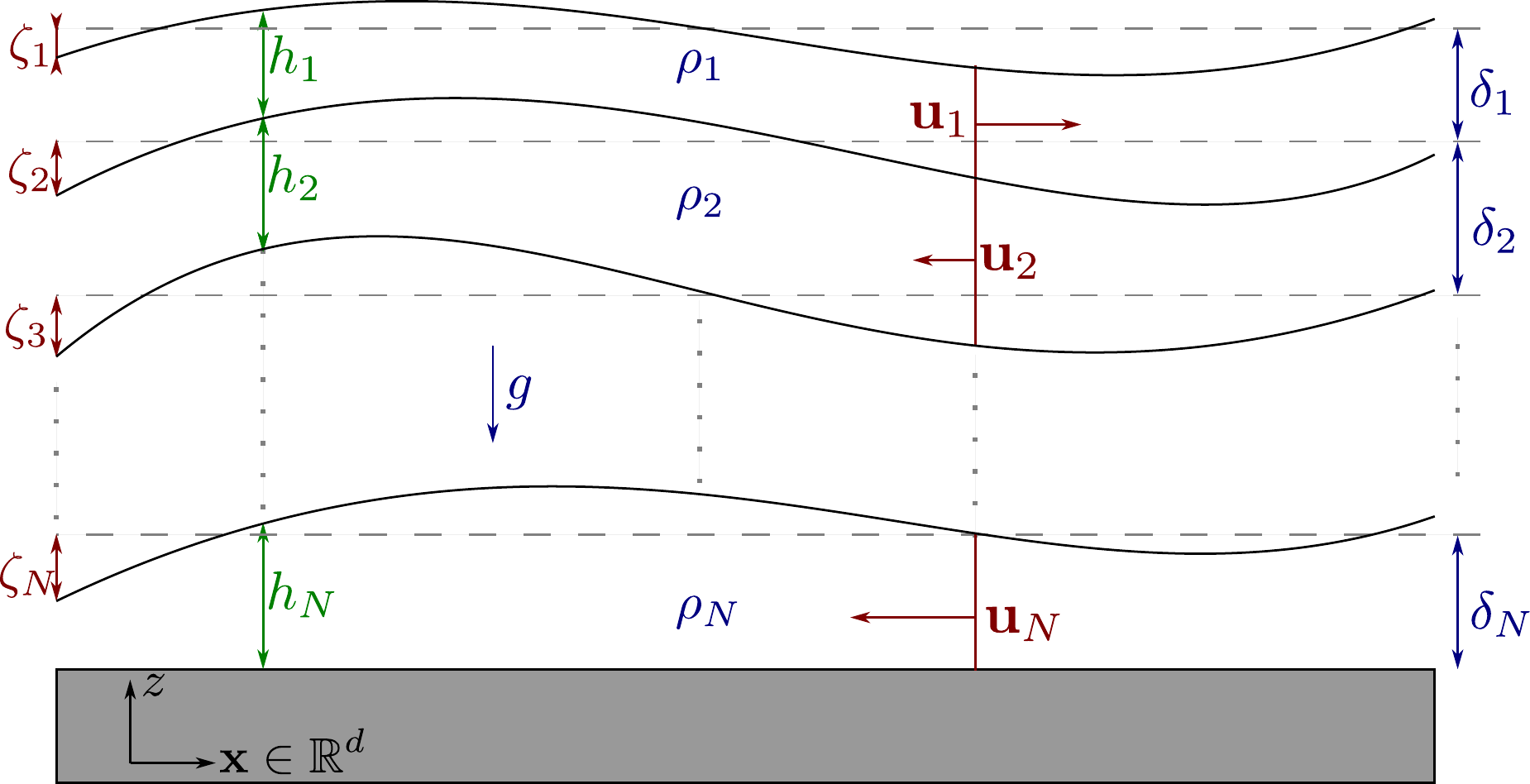}
\end{center}
\caption{Sketch of the domain and notations}
\label{F.SketchOfTheDomain}
\end{figure}

We may rescale the variables so as to replace the factor
\[ \frac{g}{\rho_n} (\rho_i-\rho_{i-1}) \longleftarrow \frac{r_i}{\gamma_n} \quad \text{ where } \quad \text{$r_i\eqdef \frac{\gamma_i-\gamma_{i-1}}{1-\gamma_1}$ and $\gamma_n\eqdef \frac{\rho_n}{\rho_N}$}.\]
 More precisely, we use the following nondimensionalization:
\[ \x \leftarrow \x / \lambda, \quad \ \zeta_n\leftarrow\zeta_n / a, \quad \delta_n \leftarrow \delta_n/ a, \quad t \leftarrow c_0 t/\lambda, \quad  \u_n \leftarrow  \u_n/ c_0,\]
where $\lambda $ is a characteristic horizontal length (say the wavelength of the flow), $a$ is a characteristic vertical length (say the typical depth of one layer at rest), and $c_0\eqdef \sqrt{(1-\gamma_1)ga}$ is a velocity. \footnote{Because we assume that the layers are all of comparable depth and the vertical stratification is balanced, in the sense that we fix $\m>0$ such that $\sup_{i\in\{1,\dots,N\}}\{\delta_i,\delta_i^{-1},r_i,r_i^{-1}\}\leq \m$,
then $c_0$ measures the typical velocity of propagation of the baroclinic modes; see Appendix~\ref{S.spectral}.}

Although $\lambda$ does not appear in our system, it plays a very important role as the Saint-Venant system may be seen as the first-order asymptotic model obtained from the full multilayer water-wave system in the limit $\mu\eqdef a^2/\lambda^2\to 0$; see~\cite{Lannes,CastroLannesa} in the one-layer case and~\cite{Duchene10} in the bilayer (albeit irrotational) setting. It may also be formally obtained using the hydrostatic and columnar motion assumptions; see~\cite{SchijfSchonfled53,Ovsjannikov79,Baines88,LiskaWendroff97,StewartDellar10}.

In this work, we ask\\
{\bf Qn: } {\em What is the behaviour of the solutions to~\eqref{FS-U-intro} in the limit $\gamma_1\to \gamma_N=1$?}
\medskip

The first observation is that the velocity evolution equations become singular, as ${r_1=\frac{\gamma_1-\gamma_0}{1-\gamma_1}\to\infty}$ since $\gamma_0=0$ by convention, so that even the existence of solutions on a non-trivial time interval is far from straightforward. 

At the linear level, it is known~\cite{StewartDellar12} (see also Appendix~\ref{S.spectral}) that the flow may be decomposed into $N$ modes, propagating as linear wave equations with distinct velocities. In our setting, the first mode, \ie the barotropic mode, propagates much faster than the other, baroclinic modes, in the sense that the former is typically of size $\approx \frac1{\sqrt{1-\gamma_1}}$ while the latter are uniformly bounded when $\gamma_1\to 1$; hence the singularity. While such a decomposition is exact only in the linear setting, we show in this work that the flow behaves in a similar way in the weak density contrast limit even when strong nonlinearities are present, provided that the initial surface deformation is small, the depth of each layer remains positive and shear velocities are not too large. Roughly speaking, we show that one may then approximate the flow as the superposition of rapidly propagating acoustic waves and a non-singular ``slow'' mode with non-trivial dynamics. 

Let us be more precise. Asymptotically, the fast mode describes the propagation of the free surface, $\zeta_1$, and total volume flux, namely $\Pi\w$ where $\w\eqdef \sum_{i=1}^N h_i\u_i$ and $\Pi\eqdef\nabla\Delta^{-1}\nabla\cdot$ is the orthogonal projection onto irrotational vector fields. One has at first order, in dimension $d=2$ and provided that $\zeta_1$ and $\Pi \w$ are initially balanced so that $\zeta_1\approx \sqrt{1-\gamma_1}\w$, the linear acoustic system
\begin{equation} \label{acoustic-intro}
\partial_t\zeta_1+\nabla\cdot (\Pi\w) =0 ; \qquad \partial_t (\Pi\w) +\frac{\sum_{n=1}^N \delta_n}{1-\gamma_1} \nabla\zeta_1=0.\end{equation}
The slow component contains all the baroclinic modes which interact strongly one with each other in the nonlinear setting. We show that it is asymptotically described by the {\em rigid-lid model}, which is obtained from the free-surface system by setting $\zeta_1\equiv 0$ in the mass conservation equations, and replacing $r_1\nabla\zeta_1$ with $\nabla p$ in the velocity evolution equations (see~\cite{Long56,Baines88}). In addition, we apply the so-called {\em Boussinesq approximation} (see \eg~\cite{Gill82}) to the limit system, that is we set $\gamma_n=1$ in the velocity evolution equations while $r_n$ remains fixed and non-trivial. This yields
\begin{equation}\label{RL}
\left\{ \begin{array}{l}
\displaystyle\partial_{ t}{\zeta_n} \ + \ \sum_{i=n}^N\nabla\cdot (h_i\u_i) \ =\ 0,  \\ 
\\
\displaystyle\partial_{ t}\u_n \  + \ \nabla p\ + \ \sum_{i=2}^n r_i \nabla\zeta_i\ + \  (\u_{n}\cdot\nabla)\u_n  \ = \ \z\ ,
\end{array}
\right. \qquad (n=1,\dots,N)
\end{equation}
with the same notation as before, except $\zeta_1(t,\x)\equiv 0$ and $h_1(t,\x)\eqdef\delta_1-\zeta_2(t,\x)$. In particular, the first equation is not an evolution equation but a constraint (conservation of total mass), namely
\[ \nabla\cdot \left(\sum_{i=1}^N  h_i\u_i\right) \ =\ 0.\]
Physically speaking, $p(t,\x)$ is (up to a constant) the pressure at the flat, rigid lid. From a mathematical viewpoint, $p$ is the Lagrange multiplier associated with the above divergence-free constraint. It may be reconstructed from the knowledge of $\zeta_2,\dots,\zeta_N,\u_1,\dots,\u_N$ by solving the Poisson equation
\[ \left( \sum_{n=1}^N \delta_n\right) \Delta p \ + \ \sum_{n=1}^N \nabla\cdot \left(   h_n(\u_n\cdot\nabla)\u_n+\u_n \nabla\cdot(h_n\u_n) + h_n \sum_{i=2}^n r_i\nabla\zeta_i \right) \ = \ 0.\]
 
 We thus offer a rigorous justification (the formal justification is generally attributed to Armi~\cite{Armi86}) of the widely used rigid-lid and Boussinesq approximations for free surface multilayer shallow-water flows with a small density contrast.

\subsection{Main results}\label{S.mr}
Before stating our main results, let us recall for convenience the system of equations at stake:
\begin{equation}\label{FS-U-mr}
\left\{ \begin{array}{l}
\displaystyle\partial_{ t}{\zeta_n} \ + \ \sum_{i=n}^N\nabla\cdot (h_i\u_i) \ =\ 0,  \\ 
\\
\displaystyle\partial_{ t}\u_n \  + \ \frac1{\gamma_n}\frac{\gamma_1}{1-\gamma_1} \nabla\zeta_1\ + \ \frac{1}{\gamma_n}\sum_{i=2}^n r_i \nabla\zeta_i\ + \  (\u_{n}\cdot\nabla)\u_n  \ = \ \z\ .
\end{array}
\right. \qquad (n=1,\dots,N)
\end{equation}
where $h_i\eqdef \delta_i+\zeta_i-\zeta_{i+1}$ with convention $\zeta_{N+1}\equiv 0$ and $r_i\eqdef \frac{\gamma_i-\gamma_{i-1}}{1-\gamma_1}$.
In the following, we fix parameters $\delta_1,\dots,\delta_N,r_2,\dots,r_N\in(0,\infty)$; and denote
\[ \m\eqdef \max_{i\in\{1,\dots,N\}}\{\delta_i,\delta_i^{-1},r_i,r_i^{-1}\}.\]
We can then reconstruct $\gamma_i\in(0,1)$ with $\gamma_i=1-\r^2 \sum_{j=i+1}^N r_j$  ($i=2,\dots,N$) and $\gamma_1=1-\r^2 $ where $\r$ is the only parameter allowed to vary, and by assumption
\[ \r  \ll 1.\]
It is also convenient to denote (notice the $\r^{-1}$ prefactor; see Remark~\ref{R.small-surface} below)
\[ \U\eqdef (\r^{-1}\zeta_1,\zeta_2,\dots,\zeta_N,u^x_1,\dots,u^x_N,u^y_1,\dots,u^y_N)^\top,\]
so that the control of $\U$ in Sobolev space $H^s(\RR^d)$ (see Appendix~\ref{S.notations} for notations) yields, when $d=2$,
\[\norm{\U}_{H^s}^2\eqdef \frac1{\r ^2}\norm{\zeta_1}_{H^s}^2+\sum_{n=2}^N \norm{\zeta_n}_{H^s}^2+\sum_{n=1}^N \norm{u^x_n}_{H^s}^2+\sum_{n=1}^N \norm{u^y_n}_{H^s}^2.\]
\medskip

Let us now state the main results of this work.
\begin{Theorem}[Large time well-posedness]\label{T.WP}
Let $s>d/2+1$ and $\U^\init\in H^s(\RR^d)^{N(1+d)}$ be such that
\begin{equation}\label{C.depth-mr}
\forall n\in \{1,\dots, N\}, \quad  \inf_{\x\in\RR^d} h_n^\init\geq h_0>0, 
\end{equation}
where $h_n^\init\eqdef \delta_n+\zeta_n^\init-\zeta_{n+1}^\init $, and we recall the convention $\zeta_{N+1}\equiv 0$.

One can set $\nu_0,\r_0^{-1}=C(\m,h_0^{-1},\norm{\U^\init}_{H^s})$ such that if $\U^\init$ satisfies additionally
\begin{equation}\label{C.hyp-mr}
\forall n\in \{2,\dots, N\}, \quad  \sup_{\x\in\RR^d}\big| \u_n^\init-\u_{n-1}^\init\big|  <\nu_0^{-1},
\end{equation}
then there exists $T>0$ and a unique $\U\in C^0(\into{T};H^s(\RR^d)^{N(1+d)})$ strong solution to~\eqref{FS-U-mr} and $\U\id{t=0}=\U^\init$, and satisfying~\eqref{C.depth-mr},\eqref{C.hyp-mr} with $h_0/2$ and $\nu_0/2$ for any $t\in\into{T}$. Moreover, one has
\[     T^{-1}\leq  \norm{\U^\init}_{H^s}\times C(\m,h_0^{-1},\norm{\U^\init}_{H^s}) \quad  \text{ and } \quad \Norm{U}_{L^\infty(0,T;H^s)}\leq  \norm{\U^\init}_{H^s}\times C(\m,h_0^{-1},\norm{\U^\init}_{H^s}) ,\]
uniformly with respect to $\r\in(0,\r_0)$.
\end{Theorem}

\begin{Theorem}[Strong convergence]\label{T.convergence}
Let $d=2$, $s>d/2+1$, and $\U^\init\in H^{s}(\RR^d)^{N(1+d)}$ as above. Then there exists $T>0$ with $  T^{-1}\leq  \norm{\U^\init}_{H^s}\times C(\m,h_0^{-1},\norm{\U^\init}_{H^s})$ and, for $\r$ sufficiently small,
\begin{itemize}
\item $\U\in C^0(\into{T};H^s(\RR^d))^{N(1+d)}$ a unique strong solution to~\eqref{FS-U-mr} and $\U\id{t=0}=\U^\init$;
\item $(p^\RL,\zeta_2^\RL,\dots,\zeta_N^\RL, \u_1^\RL,\dots,\u_N^\RL)^\top\in L^\infty(0,T;H^{s}(\RR^d))^{N(1+d)}$ a unique strong solution to~\eqref{RL} with initial data
\[ \zeta_n^\RL\id{t=0}=\zeta_n^\init \quad  ( n=2,\dots,N) \ \quad \text{ and } \quad \u_n^\RL\id{t=0}= \u_n^\init-\delta^{-1}\Pi\w^\init  \quad  ( n=1,\dots,N),\]
where $\delta\eqdef\sum_{n=1}^N \delta_n$ is the total depth, $\w^\init\eqdef\sum_{n=1}^N(\delta_n+\zeta_n^\init-\zeta_{n+1}^\init)\u_n^\init$ with convention $\zeta_1^\init=\zeta_{N+1}^\init=0$ and $\Pi\eqdef\nabla\Delta^{-1}\nabla\cdot$ the orthogonal projection onto irrotational vector fields.
\item $\r^{-1}\zeta_1^{\rm ac},\w^{\rm ac}\in C^0(\into{T};H^{s}(\RR^d))^{1+d}$ a unique strong solution to~\eqref{acoustic-intro} with
\[ \zeta_1^{\rm ac}\id{t=0}=\zeta_1^\init\quad \text{ and } \quad \w^{\rm ac}\id{t=0}=\Pi \w^\init.\]
\end{itemize}
Moreover, one has for any $0\leq s'<s$,
\[ \Norm{\U-\U^{\rm app}}_{L^\infty(0,T;H^{s'}(\RR^d))}\to 0 \qquad (\r\to0),\]
where $\U^{\rm app}\eqdef (\r^{-1}\zeta_1^{\rm ac}+\r p^\RL,\zeta_2^\RL,\dots,\zeta_N^\RL,\u_1^\RL+\delta^{-1}\Pi \w^{\rm ac}, \dots,\u_n^\RL+\delta^{-1}\Pi\w^{\rm ac} \big)^\top$.
\end{Theorem}
\begin{Remark}\label{R.small-surface}
As mentioned in the introduction, our hypotheses contain a smallness assumption on the initial deformation of the surface, namely $\zeta_1\id{t=0}=\O(\r)$. This assumptions is natural so as to balance the contributions in the (preserved in time) energy:
\[ E \eqdef  \frac12  \int_{\RR^d} \frac{\gamma_1}{\r^2}|\zeta_1|^2+\sum_{n=2}^N r_n |\zeta_n|^2+\sum_{n=1}^N \gamma_n h_n |\u_n|^2.\]
Without this assumption, the flow possesses a strongly nonlinear barotropic component, and energy methods yield a well-posedness theory over a small time-domain, $t\in\into{T}$, $T^{-1}=\O(\r^{-1})$; see Proposition~\ref{P.WP-naive} and Remark~\ref{R.naive}, below. On this timescale, the baroclinic component do not evolve, so that all the dynamics is described by the barotropic component (asymptotically as $\r\to0$).
\end{Remark}
\begin{Remark}
The requirement $\r\in(0,\r_0)$ does not lose in generality in Theorem~\ref{T.WP}: the case of non-small $\r$ follows from the standard well-posedness theory of quasilinear systems, as proved by Monjarret in~\cite{Monjarreta} and stated in Proposition~\ref{P.WP-naive}, below.
\end{Remark}
\begin{Remark} Our proof does not rely on, but rather provides, the existence and uniqueness of strong solutions of the limit (rigid-lid) system. In that respect, one may see the free-surface system~\eqref{FS-U-mr} as a penalized model for~\eqref{RL} relaxing the rigid-lid constraint. Sharper well-posedness results for the rigid-lid system in the two-layer case and without the Boussinesq approximation are provided in~\cite{GuyenneLannesSaut10,BreschRenardy11}. 
\end{Remark}
\begin{Remark}
Theorem~\ref{T.convergence} is restricted to $d=2$ because we use dispersive decay estimates on rapidly propagating acoustic waves in order to control nonlinear coupling effects between the fast and slow modes. In the case of dimension $d=1$, and provided that the initial data is sufficiently localized in space, we justified in~\cite{Duchene14a} a similar mode decomposition of the flow, by making use of the different spatial support of each mode after small time. Proposition~4.4 therein, together with Proposition~\ref{P.energy-estimate-Hs} in the present work, offer a convergence between the exact and the approximate solution with rate $\O(\r)$. The same convergence rate holds in the case of dimension $d=2$ and well-prepared initial data, in the sense of Proposition~\ref{P.well-prepared initial data}.

In the following, for the sake of simplicity, we limit our study to the case of dimension $d=2$, although we find it more telling to keep the notation $d$. The proof of Theorem~\ref{T.WP} is easily adapted to the case of dimension $d=1$.
\end{Remark}
\begin{Remark}
One could add, without any additional difficulty, a uniformly bounded and order-zero term to the system, so as to take into account for instance the Coriolis force, atmospheric pressure variations, or bottom topography. Notice however that these terms should be of size $\O(\r)$ in the evolution equation for $\zeta_1$; in particular, only small topography may be dealt with using directly our strategy. Similarly, except in the one-layer case where the the component due to Coriolis effect is an anti-symmetric perturbation of a symmetric system, one cannot allow a rapid rotation such as in the quasi-geostrophic regime, which would correspond here to $\Ro\approx \r$ where $\Ro$ is the Rossby number; see~\cite{EmbidMajda96,ReznikZeitlinBen01}. 
\end{Remark}
\begin{Remark}
Our results are valid for arbitrary $N$, but not uniformly. In particular, we cannot control the dependence of $\r_0,\nu_0$ as $N$ grows. Thus our strategy cannot be adapted to study the system in the limit $N\to\infty$, corresponding to the physically relevant situation of continuous stratification. A similar shortcoming was already noticed and discussed by Ripa in \cite{Ripa91}.
\end{Remark}

\subsection{Discussion and strategy}

A well-posedness result on system~\eqref{FS-U-mr} is stated and proved in~\cite{Monjarreta}. It follows from a standard analysis on quasilinear systems since a symbolic symmetrizer may be exhibited (see Section~\ref{S.WP-naive} below). However, due to the presence of singular components in the system, the {\em a priori} maximal time of existence of the solutions may be bounded from below only  as $T_{\max}\gtrsim\r$ using this method. Such a result is unsuitable for our purpose as the time interval shrinks to zero in the considered limit, and is inconsistent with oceanographic observations of large amplitude internal waves propagating over long distances; see \eg \cite{OstrovskyStepanyants89} and references therein.

In order to go beyond this analysis and provide a time of existence of solutions uniformly bounded from below with respect to $\r$ small, we need to take advantage of some additional structural properties of the system. This structure is put to light by a suitable change of variable, which we describe below. 

Let us introduce the shear velocities, $\v_i$, and the total horizontal momentum, $\w$, as follows:
\begin{equation} \label{UtoV} \w\eqdef  \sum_{n=1}^N h_n \u_n \quad \text{ and }\quad  \v_i\eqdef \gamma_{i}\u_{i}-\gamma_{i-1} \u_{i-1}  \quad (i=2,\dots,N) ; \end{equation}
so that, conversely, for any $n=1,\dots,N$,
\begin{equation} \label{VtoU}\u_n=\frac{\gamma_n^{-1}}{ \sum_{i=1}^N \gamma_i^{-1} h_i} \left(\w+\sum_{j=2}^N \alpha_{n,j}\v_j\right) \quad \text{ with } \alpha_{n,j}=
\begin{cases}
\sum_{i=1}^{j-1} \gamma_i^{-1} h_i& \text{ if } j\leq n,\\
-\sum_{i=j}^N \gamma_i^{-1} h_i& \text{ otherwise}.
\end{cases}
\end{equation}

One may rewrite system~\eqref{FS-U-mr} using the new variables as follows:
\begin{equation}\label{FS-V-mr}
\left\{ \begin{array}{l}
\dsp\partial_t(\r^{-1}\zeta_1)+\frac1\r\nabla\cdot\w= 0,\\ \\
\dsp\partial_t \zeta_n + \sum_{i=n}^N\nabla\cdot\left(h_i\u_i\right) = 0,  \\  \\
\dsp\partial_t \v_n  +r_n\nabla \zeta_n  + \gamma_n (\u_n\cdot \nabla)\u_n - \gamma_{n-1}  (\u_{n-1}\cdot \nabla)\u_{n-1} = \z, \\ \\
\dsp\partial_t \w+\left( \sum_{j=1}^N \gamma_{j}^{-1} h_j\right) \frac{\gamma_1}{1-\gamma_1}\nabla\zeta_1+\sum_{i=2}^N \left(\sum_{j=i}^N \gamma_{j}^{-1} h_j\right) r_i\nabla\zeta_i+\sum_{i=1}^N   \nabla\cdot \big(h_i \u_i \otimes \u_i\big)=\z,
\end{array} \right.  \hspace{-2cm}(n=2,\dots,N) 
\end{equation}
where $\nabla\cdot (h_i \u_i \otimes \u_i)\eqdef \left(\begin{smallmatrix}
\partial_x(h_i |u^x_i|^2)+\partial_y (h_i u^x_iu^y_i)\\
\partial_x(h_i  u^x_iu^y_i)+\partial_y (h_i |u^y_i|^2)
\end{smallmatrix}\right)$, and $ \u_i=(u_i^x,u_i^y)^\top$ are meant as the expressions in terms of $(\zeta_1,\dots\zeta_N,\v_2,\dots,\v_N,\w)^\top$ given in~\eqref{VtoU}.
\medskip

The above change of variables may be seen as an approximate normal form allowing to decouple the slow and fast components of the flow. Indeed,
since
\[\sum_{j=1}^N \gamma_j^{-1}h_j=\sum_{j=1}^N\delta_j+\r(\r^{-1}\zeta_1)+ \O(\r^2),\]
one sees immediately that the singular terms appear only as linear components on the evolution equations for $\r^{-1}\zeta_1$ and $\w$ ---or more precisely $\Pi\w$--- and involve only $\r^{-1}\zeta_1$ and $\Pi\w$. In other words, the leading-order terms form a system of rapidly propagating acoustic waves in $\r^{-1}\zeta_1,\w$:
\[
\left\{ \begin{array}{l}
\dsp\partial_t(\r^{-1}\zeta_1)+\frac1\r\nabla\cdot\w= 0,\\ 
\partial_t \w+ \frac1\r \left(\sum_{j=1}^N\delta_j \right) \nabla(\r^{-1}\zeta_1) =\O(1).
\end{array} \right.
\]
The remainder contains quasilinear components depending on both the fast ($\r^{-1}\zeta_1,\Pi\w$) and slow ($\zeta_2,\dots,\zeta_n,\v_2,\dots,\v_n,(\Id-\Pi)\w$) variables, so we need to consider the full system of equations~\eqref{FS-V-mr} in order to obtain the desired uniform energy estimates. 

Let us now, for the sake of simplicity, restrict our discussion to the case of dimension $d=1$. From the above, we may rewrite the system~\eqref{FS-V-mr} as
\begin{equation}\label{eq.V}
\partial_t\V + \frac1\r \B[\V]\partial_x\V=\partial_t\V + \frac1\r \L\partial_x\V+ \C[\V]\partial_x\V=0
\end{equation}
where $\V\eqdef (\r^{-1}\zeta_1,\zeta_2,\dots,\zeta_N,v_2,\dots,v_N,w)^\top$, so that $\partial_t\V + \frac1\r \L\partial_x\V=0$ represents the above acoustic wave system for the fast variables, while $\ker(\L)=\vect(\zeta_2,\dots,\zeta_N,v_2,\dots,v_N)$; and $\C[\V]$ contains lower-order (in terms of $\r$) and coupling terms.

We shall make use of the fact that one can construct a ``good'' symmetrizer of the system under the form~\eqref{eq.V}, namely we exhibit real, positive-definite matrices $\T[\V]$ such that $\T[\V]=(\T[\V])^\top$ and $\T[\V]\B[\V]=(\T[\V]\B[\V])^\top$, and satisfying the decomposition
\[ \T[\V]=\T_{(0)}+\T_{(1)}[\Pis\V]\Pis+\O(\r),\]
where $\Pis$ is the orthogonal projection onto $\ker(\L)$, the slow variables. Indeed, one obtains an energy estimate by taking the $L^2$ inner-product of the equation with $\T[\V]\V$, which only requires to estimate
\[ \norm{ \partial_x ( \frac1\r \T[\V]\B[\V]) }_{L^\infty}+\norm{ \partial_t ( \T[\V]) }_{L^\infty}.\]
Using that $\L$ and $\T_{(0)}$ are constant operators and that $\Pis\partial_t \V = - \Pis \C[\V]\partial_x\V=\O(1)$, we see that the above are estimated uniformly with respect to $\r$ small; thus we have a uniform control of the $L^2$ norm. The corresponding $H^s$ estimate with $s>d/2+1$ does not bring additional difficulties, using that $\L$ commutes with the Fourier multiplier ${\sf \Lambda}^s\eqdef (1-\Delta)^{s/2}\Id$.

There remains to understand why such symmetrizer exists for our system~\eqref{eq.V}. One could check, after tedious calculations, that the explicit one provided by Monjarret in~\cite{Monjarreta} (after applying the congruent transformation associated with the change of variables) satisfies the necessary hypotheses, but we offer in Appendix~\ref{S.spectral} an alternative and more robust construction. We show that, provided that $\V$ satisfies~\eqref{C.depth-mr},\eqref{C.hyp-mr}, then $\frac1\r \B[\V]$ has $2N$ real and distinct eigenvalues. Two of them are asymptotically equivalent to ${\lambda_\pm \approx \pm \frac1\r \sqrt{\sum_{j=1}^N\delta_j }}$ as $\r\to0$ while the other ones are uniformly bounded with respect to $\r$ small. The spectral projection corresponding to the former converge towards the projections onto the eigenspaces corresponding the two non-trivial eigenvalues of $\L$. Using the scale separation between the eigenvalues, one shows that the spectral projections corresponding to the latter are uniformly bounded, and that they converge as $\r\to 0$ towards independent, rank-one projections onto subspaces of $\ker(\L)$. Our symmetrizer is then, classically,
\[ \T[\V]\eqdef \sum_{j=1}^{2N} (\P_j[\V])^\top \P_j[\V]\]
where $\P_j[\V]$ is the spectral projection onto the $j^{\rm th}$ eigenspace of $\B[\V]$. That $\T[\V]$ enjoys the desired properties follows, using standard perturbation theory~\cite{Kato95}, from the fact that $\L$ is constant, $\ker(\L)\oplus_\perp\ran(\L)=\RR^{2N}$ and the strong scale separation between $|\lambda_\pm |\gtrsim 1/\r$ and uniformly bounded eigenvalues.

An additional difficulty arises in the situation of horizontal dimension $d=2$, due to the fact that the symmetrizer of the system ---which is constructed from the symmetrizer in dimension $d=1$ and a rotational invariance property---  is only a symbolic symmetrizer, as opposed to symmetrizers in the sense of Friedrichs. Thus we rely on para-differential calculus, but extra care must be given to ``lower order terms'' in the sense of regularity, which may effectively hurt our energy estimates if they are not uniformly bounded with respect to $\r$. As a matter of fact, we use that one can construct an explicit operator (defined as a Fourier multiplier) which symmetrizes the linear, singular contributions of the system, and use  para-differential calculus only on the next order components in terms of $\r$.

Given the uniform (with respect to $\r$ small) energy estimates, the large time well-posedness (Theorem~\ref{T.WP}) follows from the standard theory for quasilinear hyperbolic systems. The convergence results (Theorem~\ref{T.convergence} as well as additional assertions in Section~\ref{S.convergence}) proceed from rather standard techniques in the study of singular systems; see references below.

\subsection{Related earlier results}

In~\cite{Duchene14a}, the author studied the so-called inviscid bilayer Saint-Venant (or shallow water) system in the limit of small density contrast. The change of variables allowing for uniform energy estimates was exhibited therein, and convergence towards a solution of the rigid-lid limit, as well as a second-order approximation, was deduced in the case of well-prepared initial data. This work is therefore an extension of these results to the situation of (horizontal) dimension $d=2$, ill-prepared initial data as well as arbitrary number of layers.

As already noticed in the aforementioned work, our problem has many similarities with the (two-dimensional) incompressible limit for Euler equations, as studied initially in~\cite{KlainermanMajda81,KlainermanMajda82,BrowningKreiss82}. As a matter of fact, if we consider only one layer of fluid, then our problem corresponds exactly to a special case of the isentropic incompressible limit, and we recover the results of Ukai~\cite{Ukai86} and Asano~\cite{Asano87}. We will not detail the very rich history of results concerning this problem (we let the reader to \cite{Gallagher05,Masmoudi07,Alazard08} for comprehensive reviews) but rather aim at pointing out similarities and differences of our situation.

Let us first recall the two-dimensional isentropic Euler equations for inviscid, barotropic fluids:
\begin{equation}\label{Euler-isentropique}
\left\{ \begin{array}{l}
\dsp\partial_t\rho +\nabla\cdot (\rho\v )= 0,\\ 
\partial_t (\rho \v)+ \frac1{\epsilon^2} \nabla P + \nabla\cdot (\rho \v\otimes\v )=0,
\end{array} \right.
\end{equation}
where $P=P(\rho)$ is a given pressure law, $\rho>0$ is the density, $\v$ the velocity, and $\epsilon$ the dimensionless Mach number. As claimed above, one recognizes exactly~\eqref{FS-V-mr} in the one layer setting ($N=1$), by setting $P(\rho)\propto\rho^2$, and identifying 
\[ \rho \longleftrightarrow h_1=\delta_1+\zeta_1\ , \quad \rho\v \longleftrightarrow \w\quad \text{ and } \quad \epsilon \longleftrightarrow \r.\]
Of course, the difficulty in our case is that, as one considers additional layers of fluids, these equations are coupled with additional equations on additional unknowns, so as to produce a full quasilinear system. Since these additional equations are non-singular with respect to the small parameter, it is tempting to compare our situation with the incompressible limit for the non-isentropic Euler
equations, where~\eqref{Euler-isentropique} is coupled with an additional evolution equation for the entropy, $S$:
\[ \partial_t S+(\v\cdot\nabla)S=0\]
and $P=P(\rho,S)$. 

Our situation, however, is quite different. This can be seen from the fact that, contrarily to the non-isentropic Euler equations, the linearized system is balanced, in the sense that a small perturbation of the ``slow'' component of the reference state induces only a small deviation for the solution. In other words, using the notation of the above discussion, the symmetrizer of the non-isentropic Euler equations does not satisfy $\T[\V]=\T_{(0)}+\T_{(1)}[\Pis\V]\Pis+\O(\r)$ but only $\T[\V]=\T_{(0)}+\T_{(1)}[\Pis\V]+\O(\r)$; see discussion in~\cite{MetivierSchochet01}. 
This additional property in our situation allows in particular to straightforwardly deduce $H^s$ energy estimates from the corresponding $L^2$ energy estimate; and to obtain the strong convergence result of Theorem~\ref{T.convergence} simply from dispersive estimates on the ``acoustic'' component of the flow, as originally carried out by  Ukai~\cite{Ukai86} and Asano~\cite{Asano87} in the isentropic case.

In order to deal with this situation, Métivier and Schochet~\cite{MetivierSchochet01} (see also~\cite{BreschMetivier10}) rely on the fact that their system enjoys a diagonal block structure and that the symmetrizer commutes exactly with the singular operator, denoted $\frac1\r\L\partial_x $ in the above discussion. Roughly speaking, this means that we may control the compressible and isentropic component of the flow independently of the acoustic component by simply projecting onto $\ker(\L)$. Since such assumptions are only approximately satisfied in our situation, our system is rather related to the ``$\r$-balanced'' (and not ``$\r$-diagonal'') systems studied by Klainerman and Majda~\cite{KlainermanMajda81}, although we do not restrict ourselves to well-prepared initial data.

In this spirit, our proofs rely as little as possible on explicit calculations, thus we expect that the general strategy may be successfully applied to other situations and other frameworks, such as the ones presented in~\cite{ParisotVila15}. This is why we use mostly in the following the terminology of ``fast {\em vs} slow'' mode/component instead of ``barotropic {\em vs} baroclinic'' which is more relevant to our initial oceanographic motivation (see~\cite{Gill82}); or ``acoustic {\em vs} incompressible'' associated with Euler equations.

\subsection{Outline of the paper}

In Section~\ref{S.WP-naive}, we recall that our quasilinear system admits an explicit (symbolic) symmetrizer, which yields immediately a well-posedness theory for~\eqref{FS-U-mr}, for any fixed $\r>0$. 

In Section~\ref{S.normal-form}, we exhibit the structural properties enjoyed by our system, after the change of variables~\eqref{UtoV},\eqref{VtoU}, and which allow the uniform (with respect to $\r$ small) energy estimates provided in Section~\ref{S.energy-estimates-L2} (Proposition~\ref{P.energy-estimate-L2}) and Section~\ref{S.energy-estimates-Hs} (Proposition~\ref{P.energy-estimate-Hs}). 

We deduce in Section~\ref{S.WP} the large-time well-posedness result of Theorem~\ref{T.WP}. Additionally, we show in Proposition~\ref{P.well-prepared initial data} that an assumption of well-prepared initial data is propagated by the flow for positive times.

Section~\ref{S.convergence} is dedicated to convergence results. We first state in Proposition~\ref{P.weak-convergence} a weak convergence result for the solutions of the free-surface system as $\r\to0$. As in the incompressible limit for Euler equations, the convergence cannot be strong uniformly in time, due to the rapidly propagating fast mode. However, we show in Proposition~\ref{P.convergence-WP} that this small initial layer in time vanishes in the case of well-prepared initial data, and then characterize the defect for general initial data in Section~\ref{S.convergence-strong}, yielding Theorem~\ref{T.convergence}.

Appendix~\ref{S.notations} contains a description of some notations used throughout the text, as well as a short review of standard results concerning product and commutator estimates in Sobolev spaces (Section~\ref{S.product}) and Bony's paradifferential calculus (Section~\ref{S.para}). 

Finally, Appendix~\ref{S.spectral} is dedicated to some results on the eigenstructure of our system, which are used in Section~\ref{S.normal-form}.

\bigskip

\section{Standard well-posedness theory}\label{S.WP-naive}

In this section, we fix $\r>0$ and construct an explicit (symbolic) symmetrizer of~\eqref{FS-U-mr}. This offers a well-posedness result similar to Theorem~\ref{T.WP}, although non-uniformly with respect to $\r$ small. This analysis has been provided by Monjarret in~\cite{Monjarreta}; we recall it here for the sake of completeness and because the objects defined therein will be of later use.
\medskip

Let us first rewrite the free-surface system~\eqref{FS-U-mr} with a matricial, compact formulation. Provided that $\U\eqdef (\r^{-1}\zeta_1,\zeta_2,\dots,\zeta_N,u^x_1,\dots,u^x_N,u^y_1,\dots,u^y_N)^\top\in C^0(\into{T};H^s(\RR^d))^{N(1+d)}$ with ${s>d/2+1}$, one can rewrite~\eqref{FS-U-mr} equivalently as
\begin{equation}\label{FS-U-compact}
\partial_t \U+\frac1\r\A^x[\U]\partial_x \U+\frac1\r\A^y[\U]\partial_y \U= \z ,
\end{equation}
with
\begin{equation} \label{def-A}\frac1\r\A^x[\U]\eqdef
\left(\begin{MAT}(b){c:c:c}
{\sf M}(u^x) & {\sf H}(\zeta) &{\sf 0}_N\\:
{\sf R} & \D(u^x)&{\sf 0}_N \\:
{\sf 0}_N &{\sf 0}_N& \D(u^x) \\
\end{MAT}\right), \qquad \frac1\r\A^y[\U]\eqdef
\left(\begin{MAT}(b){c:c:c}
{\sf M}(u^y)  &{\sf 0}_N& {\sf H}(\zeta) \\:
{\sf 0}_N  & \D(u^y) &{\sf 0}_N\\:
{\sf R}  &{\sf 0}_N& \D(u^y) \\
\end{MAT}\right).
\end{equation}
Here and thereafter, we heavily make use of the block structure of $N(1+d)$-by-$N(1+d)$ matrices. We denote by ${\sf 0}_N$ the $N$-by-$N$ matrix with only zero entries, and for $u\in \RR^N$, $\D(u)=\diag(u_1,\dots,u_N)$. Moreover, ${\sf M},{\sf H}$ are upper-triangular and ${\sf R}$ is lower-triangular and are defined by
\[ {\sf M}(u)_{n,i}= 
\begin{cases} 
\frac1\r (u_i-u_{i-1})  & \text{ if } 1=n<i, \\
u_i-u_{i-1}  & \text{ if } 2\leq n < i, \\
 u_i  & \text{ if } i = n, \\
0 & \text{ if } i< n ,
 \end{cases} 
 \qquad 
\begin{array}{l}  {\sf H}_{n,i}= 
\begin{cases} 
\frac1\r h_i  & \text{ if } n=1, \\
h_i & \text{ if } i\geq n \geq 2,\\
0 & \text{ if } i< n ,
 \end{cases} \\ \\
  {\sf R}_{n,i}= 
\begin{cases} 
\frac1\r \frac{\gamma_1}{\gamma_n}  & \text{ if } i =1, \\
\frac{r_i}{\gamma_n} & \text{ if } n \geq i \geq 2,\\
0 & \text{ if }  n<i.
 \end{cases}
 \end{array}
 \]

\begin{Proposition}[\cite{Monjarreta},~Theorem 2.8]\label{P.WP-naive}
Let $s>d/2+1$ and $\U^\init\in H^s(\RR^d)^{N(1+d)}$ be such that
\begin{equation}\label{C.depth-U}
\forall n\in\{1,\dots,N\},\quad  \inf_{\x\in\RR^d} h_n = \inf_{\x\in\RR^d} \big(\delta_n+\zeta_n-\zeta_{n+1} \big)  > h_0>0,
\end{equation}
where we recall that $\zeta_{N+1}=0$ by convention. 
One can set $\nu_0=C(\m,h_0^{-1})>0$ such that if $\U^\init$ satisfies additionally
\begin{equation}\label{C.hyp-U}
\forall n\in\{2,\dots,N\}, \quad \sup_{\x\in\RR^d} \norm{u_n^x-u_{n-1}^x}+ \norm{u_n^y-u_{n-1}^y} < \nu_0^{-1},
\end{equation}
then there exists a unique $T_{\max}>0$ and $\U\in C^0(\into{T_{\max}};H^s(\RR^d)^{N(1+d)})$, maximal solution to~\eqref{FS-U-compact} and $\U\id{t=0}=\U^\init$.

Moreover, if $T_{\max}<\infty$, then $\norm{\U}_{W^{1,\infty}}(t)\to \infty $ ($t\uparrow T_{\max}$) or one of the hyperbolicity conditions~\eqref{C.depth-U},\eqref{C.hyp-U} ceases to be true.
\end{Proposition}
\begin{Remark} \label{R.naive}
Naively following the above strategy and keeping track of the dependency of constants with respect to the parameter $\r$ would yield a disappointing lower bound on the maximal time of existence, namely $T_{\max}^{-1}\lesssim \r^{-1}$, even when the initial surface deformation is assumed small as in Theorem~\ref{T.WP}. It is the main result of this work that, in this case, the well-posedness theory and uniform energy estimates can be extended to a non-shrinking time domain as $\r\to0$.

The conditions~\eqref{C.depth-U},\eqref{C.hyp-U} ensure that the symmetrizer, defined in~\eqref{def-Sx} below, is coercive. It is therefore a sufficient condition for hyperbolicity. Except in very specific cases, one has very few information on the domain of hyperbolicity; see discussion in Appendix~\ref{S.spectral}. 

The fact that~\eqref{C.hyp-U} requires a control on the shear velocities, $\u_n-\u_{n-1}$ rather than on the velocities themselves is allowed by some freedom in the choice of the symmetrizer. Notice in particular that the Hessian of the conserved energy yields a natural symmetrizer for our system of conservation laws in the case of irrotational flows, but it does not enjoy the desired property.

The ability to construct symmetrizers depending strongly on the shear velocities but only weakly on a background velocity (or on the total volume flux)  is also essential for us to obtain results outside of the scope of well-prepared initial data; see~\cite{Duchene14a} for an analysis where this property is not used.
\end{Remark}
\begin{proof}
Let us introduce the symbol of~\eqref{FS-U-compact}:
\[\A[\U,\xi]\eqdef \xi^x \A^x[\U]+\xi^y \A^y[\U].\]
An important property of our system is that is satisfies {\em rotational invariance}~\cite[Section~1.2]{Monjarreta}. More precisely, one can easily check that
\begin{equation}\label{A-rot}
\forall \xi=(\xi^x,\xi^y)^\top\in\RR^d\setminus\{\z\}, \qquad \A[\U,\xi]=\Q(\xi)^{-1}\A^x[\Q(\xi)\U]\Q(\xi)|\xi|,
\end{equation}
where
\begin{equation} \label{def-Q}
\Q(\xi)= \frac1{|\xi|}
\left(\begin{MAT}(b){c:c:c}
|\xi|{\sf I}_N & {\sf 0}_N & {\sf 0}_N \\:
 {\sf 0}_N & \xi^x {\sf I}_N & \xi^y{\sf I}_N  \\:
{\sf 0}_N &-\xi^y{\sf I}_N &  \xi^x {\sf I}_N \\
\end{MAT}\right).
\end{equation}
where ${\sf I}_N$ is the $N$-by-$N$ identity matrix and $|\xi|\eqdef (|\xi^x|^2+|\xi^y|^2)^{1/2}$. Obviously, $\Q(\xi)$ is homogeneous of degree $0$ in $\xi$, with entries in $C^\infty(\RR^d\setminus\{\z\})$ and is orthogonal:  $\Q(\xi)^{-1}=\Q(\xi)^\top$.

This allows to construct a (symbolic) symmetrizer of the system from a (Friedrichs) symmetrizer of $\A^x$ alone.
More precisely, define
\begin{equation} \label{def-Sx}
\S^x[\U]= \left(\begin{MAT}(b){c:c:c}
{\sf D}(\t r) & {\sf L}^\top& {\sf 0}_N \\:
{\sf L}&  {\sf D(\gamma)}{\sf D}(h) & {\sf 0}_N  \\:
{\sf 0}_N &{\sf 0}_N &  {\sf D(\gamma)}{\sf D}(h)  \\
\end{MAT}\right)
\end{equation}
with $\D(\t r)\eqdef\diag(\gamma_1,r_2,\dots,r_N)$ and
${\sf L}\eqdef-{\sf D}[\gamma (u^x+K^x)]{\sf \Delta}{\sf D}[\e_\r] $ where $\D(\e_\r)\eqdef\diag(\r,1,\dots,1)$, 
\[  {\sf \Delta} \eqdef\begin{pmatrix} -1 & 1 & & \text{\large (0)}  \\   & \ddots &  \ddots  & \\ &  & \ddots & 1  \\ \text{\large (0)}  &&  & -1 \end{pmatrix}\ , \qquad {\sf \Delta}^{-1} =\left(\begin{MAT}(b){ccc}
 -1 & & \text{\large (-1)} \\   & \ddots &  \\ \text{\large (0)} &  & -1 \\
 \end{MAT} \right)\ ,\]
and the parameter $K^x$ may be chosen freely in $\vect(u^x)$.

That $\S^x[\U]\A^x[\U]$ is symmetric is easily checked once we notice the following identities in~\eqref{FS-U-compact}: ${\sf M}(u)=\D(\e_\r)^{-1}{\sf \Delta}^{-1}\D(u){\sf \Delta}\D(\e_\r)$, ${\sf H}=-\D(\e_\r)^{-1}{\sf \Delta}^{-1}\D(h)$ and ${\sf R}=-\D(\gamma^{-1})({\sf \Delta}^{-1})^\top\D(\t r)\D(\e_\r)^{-1}$. 

It follows that $\S[\U,\xi]\eqdef \Q(\xi)^{-1}\S^x[\Q(\xi)\U]\Q(\xi)$ is a symbolic symmetrizer of~\eqref{FS-U-compact}, since
\[ \S[\U,\xi]\A[\U,\xi]=\Q(\xi)^\top\S^x[\Q(\xi)\U]\A^x[\Q(\xi)\U]\Q(\xi)|\xi| \]
is obviously symmetric. Finally, one may choose $K^x=-u^x_1$ for instance, so that as soon as $\U$ satisfies~\eqref{C.depth-U},\eqref{C.hyp-U} with $\nu_0$ sufficiently small, then $\S^x[\Q(\xi)\U]$ is strictly diagonally dominant with positive diagonal entries, and therefore definite positive, uniformly for any $\xi\in\RR^d\setminus\{\z\}$.

Proposition~\ref{P.WP-naive} is now a direct consequence of the standard theory of first-order hyperbolic quasilinear systems; see~\cite{M'etivier08} for instance.
\end{proof}

\section{Uniform energy estimates}\label{S.energy-estimates}
In this section, we establish uniform (with respect to $\r$ small) energy estimates, which are the essential ingredients in the proof of our main results. We first exhibit in Section~\ref{S.normal-form} some properties of the system obtained after the change of variables~\eqref{UtoV}-\eqref{VtoU}, and which allow the $L^2$ energy estimate of Proposition~\ref{P.energy-estimate-L2} (Section~\ref{S.energy-estimates-L2}) and in turn the $H^s$ energy estimate of Proposition~\ref{P.energy-estimate-Hs} (Section~\ref{S.energy-estimates-Hs}).

\subsection{A new formulation}\label{S.normal-form}

In what follows, we fix $h_0>0$ and always assume that
\begin{equation}\label{C.depth}  \forall n\in\{1,\dots,N\},\quad  h_n \eqdef \delta_n+\zeta_n-\zeta_{n+1}\geq h_0>0
\end{equation} 
(recall the convention: $\zeta_{N+1}=0$). More precisely, we work with $\VV\subset \RR^{N(1+d)}$
defined by
\[
\VV=\Big\{  (\r^{-1}\zeta_1,\zeta_2,\dots,\zeta_N,\star,\dots,\star)^\top\in \RR^{N(1+d)} \text{ such that~\eqref{C.depth} holds}\Big\}.
\]
Consequently, the change of variables~\eqref{UtoV}-\eqref{VtoU} define self-homeomorphisms between
\begin{align*} \U&\eqdef (\r^{-1}\zeta_1,\zeta_2,\dots,\zeta_N,u^x_1,\dots,u^x_N,u^y_1,\dots,u^y_N)^\top\in \VV \quad \text{ and }\\
\V&\eqdef (\r^{-1}\zeta_1,\zeta_2,\dots,\zeta_N,v^x_2,\dots,v^x_N,w^x,v^y_2,\dots,v^y_N,w^y)^\top\in \VV,
\end{align*}
and we denote
\[ F:\left\{
\begin{array}{ccl}
\VV & \to & \VV\\
\V  & \mapsto & \U\eqdef F(\V)
\end{array}\right. \quad \text{ and } \quad F^{-1}:\left\{\begin{array}{ccl}
\VV & \to & \VV\\
\U  & \mapsto & \V\eqdef F^{-1}(\U)
\end{array}\right. .\]

Consider the Jacobian matrix associated to $F^{-1}$:
\begin{equation} \label{def-J-1} 
\J^{F^{-1}}[\U]\eqdef
\left(\begin{MAT}(b){c:c:c}
{\sf I}_N & {\sf 0}_N& {\sf 0}_N \\:
{\sf C}(u^x)&  {\sf \Delta}_h(\zeta) & {\sf 0}_N  \\:
{\sf C}(u^y) &{\sf 0}_N & {\sf \Delta}_h(\zeta)  \\
\end{MAT}\right) 
\end{equation}
where
\[ 
{\sf \Delta}_h(\zeta)=\left(\begin{MAT}{cccc}
-\gamma_1 &\gamma_2  & & \text{\large (0)} \\
 & \ddots &  \ddots & \\
  \text{\large (0)} &  & -\gamma_{N-1} & \gamma_N \\:
  h_1 & \cdots & \cdots & h_N \\
\end{MAT}\right)
\quad \text{ and } \quad 
{\sf C}(u)=\left(\begin{MAT}(r){cccc} 
& & & \\
 &  & \text{\large (0)} & \\
 &  &  &  \\:
\r u_1 & u_2-u_1 & \cdots & u_N-u_{N-1} \\
  \end{MAT}\right).
\]
From the inverse function theorem, one has
\begin{equation} \label{def-J} 
\J^{F}[\V]= (\J^{F^{-1}}[F(\V)])^{-1} 
\end{equation}
and
\[
 (\J^{F^{-1}}[\U])^{-1} =
\left(\begin{MAT}(b){c:c:c}
{\sf I}_N & {\sf 0}_N& {\sf 0}_N \\:
-{\sf \Delta}_h(\zeta)^{-1} {\sf C}(u^x)&  {\sf \Delta}_h(\zeta)^{-1} & {\sf 0}_N  \\:
-{\sf \Delta}_h(\zeta)^{-1} {\sf C}(u^y) &{\sf 0}_N & {\sf \Delta}_h(\zeta)^{-1}  \\
\end{MAT}\right) 
\]
where (as is easier seen directly from~\eqref{VtoU}),
\[ \left({\sf \Delta}_h(\zeta)^{-1}\right)_{n,j} =\frac{\gamma_n^{-1}}{ \sum_{i=1}^N \gamma_i^{-1} h_i} \alpha_{n,j+1}\quad \text{ with } \alpha_{n,j}=
\begin{cases}
1& \text{ if }  j=N+1 ,\\
\sum_{i=1}^{j-1} \gamma_i^{-1} h_i& \text{ if } 2\leq j\leq n,\\
-\sum_{i=j}^N \gamma_i^{-1} h_i& \text{ otherwise}.
\end{cases} \]

Applying the change of variables $\U=F(\V)$ in~\eqref{FS-U-compact} yields, for sufficiently regular functions (see Lemma~\ref{L.UtoV} below),
\begin{equation}\label{FS-V-compact}
\partial_t \V+\frac1\r\B^x[\V]\partial_x \V+\frac1\r\B^y[\V]\partial_x \V=\z,
\end{equation}
which we can identify with~\eqref{FS-V-mr}, and where $\B^x[\V],\B^y[\V]$ are explicitly given in terms of $\A^x[\V],\A^y[\V]$ (displayed in~\eqref{def-A}) and $\J^F[\V] $ through
\begin{equation}\label{AtoB}
\B^x[\V] = (\J^F[\V])^{-1}\A^x[F(\V)] \J^F[\V] \quad \text{ and } \quad \B^y[\V] = (\J^F[\V])^{-1}\A^y[F(\V)] \J^F[\V] .
\end{equation}

Finally, we denote by $\Pif,\Pifx$ the orthogonal projection onto the ``fast'' variables, namely 
\begin{equation} \label{def-Pi} 
\Pif \eqdef \left(\begin{MAT}(c){c:c:c}
\D(\e_1)& {\sf 0}_N& {\sf 0}_N \\:
 {\sf 0}_N&\D(\e_N)&  {\sf 0}_N\\:
  {\sf 0}_N& {\sf 0}_N& \D(\e_N)\\
\end{MAT}\right), \quad \Pifx \eqdef \left(\begin{MAT}(c){c:c:c}
\D(\e_1)& {\sf 0}_N& {\sf 0}_N \\:
 {\sf 0}_N&\D(\e_N)&  {\sf 0}_N\\:
  {\sf 0}_N& {\sf 0}_N& {\sf 0}_N\\
\end{MAT}\right),
\end{equation}
with $\D(\e_1)\eqdef  \diag(1,0,\dots,0)$ and $\D(\e_N)\eqdef  \diag(0,\dots,0,1)$.

The following result is now straightforward.
\begin{Lemma}\label{L.J}
Let $\U=F(\V)\in \VV$, and recall the definition of $\Q(\xi)$ in~\eqref{def-Q}.
Then 
\begin{equation}\label{J-rot}
\forall \xi\in\RR^d\setminus\{\z\}, \qquad \Q(\xi)F(\V)=F(\Q(\xi)\V)\quad \text{ and } \quad \Q(\xi)\J^F[\V]= \J^F[\Q(\xi)\V]\Q(\xi).
\end{equation}
Moreover, $\J^F[\V],(\J^F[\V])^{-1}:\VV\to \M_{N(1+d)}(\RR)$ are well-defined and smooth and satisfy
\begin{subequations}\label{J-estimates}
\begin{gather}\label{J-estimate-1}
 \norm{F(\V)}\leq C(\m,h_0^{-1},\norm{\V}) \norm{\V}, \qquad \norm{F^{-1}(\U)}\leq C(\m,h_0^{-1},\norm{\U}) \norm{\U};\\\label{J-estimate-2}
\Norm{\J^F[\V]} \leq C(\m,h_0^{-1},\norm{\V}) \quad ; \quad \Norm{(\J^F[\V])^{-1}} \leq  C(\m,h_0^{-1},\norm{\V}) ;\\\label{J-estimate-3}
\Norm{\J^F[\V]-\J^F[\Pis\V]} \leq \r\ C(\m,h_0^{-1},\norm{\V}) \norm{\V}.
\end{gather}
\end{subequations}
\end{Lemma}
\begin{proof}
Only the last estimate requires an explanation. Remark that the first variable of $\V$ contributes to ${\sf \Delta}_h(\zeta)^{-1}$ only through $h_1=\delta_1-\zeta_2+\r\frac{\zeta_1}{\r}$. Thanks to the $\r$ prefactor, one deduces that
\[ \Norm{{\sf \Delta}_h(\zeta)^{-1} -{\sf \Delta}_h(\Pis\zeta)^{-1} } \leq  \r\ C(\m,h_0^{-1},\norm{\zeta}) \norm{\zeta_1} , \]
where we denoted $\Pis\zeta\eqdef (0,\zeta_2,\dots,\zeta_N)^\top$. Similarly,  by~\eqref{UtoV}, one has
\[ u_i-u_{i-1}=v_i+ (1-\gamma_i) u_i +  (1-\gamma_{i-1}) u_{i-1} \]
and, by definition, $1-\gamma_i=\r^2 \sum_{j=i+1}^N r_j$, so that
\[ \Norm{\C(u(\V))-\C(u(\Pis \V)) } \leq  \r\ C(\m,h_0^{-1},\norm{\V}) \norm{\V} ,\]
where $u(\V)$ represents the velocity variables of $F^{-1}(\V)$, as given by~\eqref{VtoU}.
The result is now clear.
\end{proof}

\begin{Lemma}\label{L.B}
The functions $\B^x,\B^y:\VV\to \M_{N(1+d)}(\RR)$ are well-defined and smooth and satisfy the rotational invariance
\begin{equation}\label{B-rot}
\forall \xi\in\RR^d\setminus\{\z\}, \qquad  \B[\V,\xi]\eqdef \xi^x \B^x[\V]+\xi^y \B^y[\V]=\Q(\xi)^{-1}\B^x[\Q(\xi)\V]\Q(\xi)|\xi|;
\end{equation}
as well as the following estimates:
\begin{subequations}\label{B-estimates}
\begin{gather}
\Norm{\B^x[\V]}\leq C(\m,h_0^{-1},\norm{\V}), \quad \Norm{\Pisx\B^x[\V]}\leq \r\  C(\m,h_0^{-1},\norm{\V});
 \label{B-estimates-1}\\
\Norm{\B^x[\V_1]-\B^x[\V_2]}\leq \r\ C(h_0^{-1} ,\m,\norm{\V_1},\norm{\V_2})\norm{\V_1-\V_2}.
 \label{B-estimates-4}
\end{gather}
\end{subequations}
\end{Lemma}
\begin{proof}
Identity~\eqref{B-rot} is deduced from~\eqref{AtoB},~\eqref{A-rot} and~\eqref{J-rot}:
\begin{align*} 
 \xi^x \B^x[\V]+\xi^y \B^y[\V]&=  (\J^F[\V])^{-1} (\xi^x \A^x[F(\V)] + \xi^y \A^y[F(\V)] ) \J^F[\V] \\
 &= (\J^F[\V])^{-1}  \Q(\xi)^{-1}\A^x[\Q(\xi)F(\V) ]\Q(\xi) \J^F[\V]|\xi|\\
 &= \Q(\xi)^{-1}(\J^F[\Q(\xi)\V])^{-1}\A^x[F(\Q(\xi) \V) ]\J^F[\Q(\xi)\V]\Q(\xi)|\xi|\\
 &= \Q(\xi)^{-1}\B^x[\Q(\xi)\V]\Q(\xi)|\xi|.
 \end{align*}

Estimates~\eqref{B-estimates} may be deduced from identity~\eqref{AtoB} and the explicit expressions of $\A^x[F(\V)]$ and $\J^F[\V] ,\  (\J^F[\V])^{-1}$ given in~\eqref{def-A},\eqref{def-J},\eqref{def-J-1}.  They are also apparent when identifying~\eqref{FS-V-compact} with~\eqref{FS-V-mr}. In particular, one immediately sees that the evolution equations for $\zeta_2,\dots,\zeta_N,\v_2,\dots,\v_N$ are non-singular (for $\r$ small), and the singular term on the evolution equation for $w^y$ involves only $\partial_y\zeta_1$, so that
\[ \Norm{\frac1\r\Pisx \B^x[\V]}\leq   C(\m,h_0^{-1},\norm{\V}) \]
and
\[
\Norm{\frac1\r\Pisx\B^x[\V_1]-\frac1\r\Pisx\B^x[\V_2]}\leq \ C(h_0^{-1} ,\m,\norm{\V_1},\norm{\V_2})\norm{\V_1-\V_2}.
\]
The only singular terms (for $\r$ small) arise from the first and last equations, which read
\[ \partial_t(\r^{-1}\zeta_1)+\r^{-1}\nabla\cdot \w =0 \quad \text{ and } \quad   \partial_t \w+\left(\sum_{j=1}^N \gamma_j^{-1}h_j\right)\frac{\r\gamma_1}{1-\gamma_1} \nabla(\r^{-1}\zeta_1) =r_1^x[\V]\partial_x\V+r_1^y[\V]\partial_y\V ,\]
where $r_1^x,r_1^y$ are smooth and enjoy the same estimates as $\frac1\r\Pis \B^x$ above.
We now only need to remark that $\frac{\r\gamma_1}{1-\gamma_1}=\gamma_1\r^{-1}$ (by definition) and, since $1-\gamma_i=\r^2 \sum_{j=i+1}^N r_j$,
\[\sum_{j=1}^N \gamma_j^{-1}h_j=\sum_{j=1}^N h_j+\sum_{j=1}^N \frac{1-\gamma_j}{\gamma_j} h_j=\sum_{j=1}^N\delta_j+ \r (\r^{-1}\zeta_1)+\r^2 r_2(\V),\]
where, again, $r_2$ is smooth and uniformly estimated as above. Lemma~\ref{L.B} is now straightforward.
\end{proof}

Finally, let us provide some estimates on the symbolic symmetrizer of the system.
\begin{Lemma}\label{L.T}
There exists $\T^x:\VV\to \M_{N(1+d)}(\RR)$ a smooth function such that 
\[\forall \V\in \VV, \quad  (\T^x[\V])^\top=\T^x[\V]\quad \text{ and } \quad (\T^x[\V]\B^x[\V])^\top=\T^x[\V]\B^x[\V].\]
Moreover, one can set $\r_0^{-1},\nu=C(\m,h_0^{-1},\norm{\V})>0$ such that if $\r\in(0,\r_0)$ and $\V,\V_1,\V_2$ satisfy
\begin{equation}\label{C.hyp-V}
\forall n\in\{2,\dots,N\}, \quad \sup_{\x\in\RR^d} \norm{v^x_n} +\norm{v^y_n} < \nu^{-1}\ ,
\end{equation}
then one has the following estimates:
\begin{subequations}\label{T-estimates}
\begin{gather}
c_0 \Id \leq \T^x[\V] , \quad  \Norm{\T^x[\V]}\leq C_0, \quad \text{ with } c_0^{-1}=N(1+d),\ C_0=C(\m,h_0^{-1},\norm{\V});
\label{T-estimates-1}\\
\Norm{\T^x[\V_1]-\T^x[\V_2]}+\frac1\r \Norm{(\T^x[\V_1]-\T^x[\V_2])\Pifx}\leq  C(h_0^{-1} ,\m,\norm{\V_1},\norm{\V_2}) \norm{\V_1-\V_2};
\label{T-estimates-2}\\
 \Norm{\T^x[\V_1]-\T^x[\V_2]}\leq  C(h_0^{-1} ,\m,\norm{\V_1},\norm{\V_2}) \left(\norm{\Pis(\V_1-\V_2)}+\r \norm{\V_1-\V_2}\right) .
\label{T-estimates-3}
\end{gather}
\end{subequations}
\end{Lemma}
\begin{proof}
One could define the symmetrizer as
$
\T^x[\V]=  (\J^F[\V])^\top \S^x[F(\V)] \J^F[\V],
$
where $\S^x[F(\V)]$ is the symmetrizer associated with $\A^x[F(\V)]$ and has been displayed in~\eqref{def-Sx}, and check the properties directly on this explicit symmetrizer. Some of the estimates, however, rely on delicate cancellations which have no obvious explanation using this method. This is why we find it more instructive to construct our symmetrizer using the spectral properties of our system, whose study we postpone to Appendix~\ref{S.spectral} for the sake of readability. It is proved in Lemmas~\ref{L.B-trivialeigen} and~\ref{L.B-eigen} that $\B^x[\V]$ has only real and semisimple eigenvalues:
\[ \frac1\r \B^x[\V] = \sum_{n=1}^N \mu_{n}[\V]  \P_{ n}[\V]+ \mu_{-n}[\V] \P_{-n}[\V] +u^x_n[\V]  \P^0_n[\V],\]
where $\mu_{\pm n}[\V] ,u^x_n[\V] \in \RR$, and $\P_{ \pm n}[\V], \P^0_n[\V]$ are rank-one spectral projections. We now define
\begin{equation}\label{def-T}
 \T^x[\V] = \sum_{n=1}^N (\P_{ n}[\V])^\top \P_{ n}[\V]+ (\P_{ -n}[\V])^\top  \P_{-n}[\V]+(\P^0_n[\V])^\top  \P^0_n[\V].
 \end{equation}
Indeed, $\T^x[\V]$ is obviously symmetric, and so is
\[ \frac1\r\T^x[\V]\B^x[\V]= \sum_{n=1}^N \mu_{n}[\V] (\P_{ n}[\V])^\top \P_{ n}[\V]+ \mu_{-n}[\V]  (\P_{ -n}[\V])^\top  \P_{-n}[\V] +u^x_n[\V]  (\P^0_n[\V])^\top  \P^0_n[\V],\]
and one has for any $\W\in\RR^{N(1+d)}$, 
\[  \big(\T^x[\V] \W,\W\big)=\sum_{n=1}^N \norm{ \P_{ n}[\V]\W}^2+\norm{  \P_{ -n}[\V]\W}^2+\norm{ \P^0_n[\V]\W}^2\geq \frac1{(N(1+d))^2} \norm{\W}^2.\]
The upper bound in~\eqref{T-estimates-1} follows from 
\[  \sum_{n=1}^N\Norm{\P_{ n}[\V]}+\Norm{\P_{ -n}[\V]}+\Norm{\P^0_{ n}[\V]} \leq C(\m,h_0^{-1},\norm{\V}),\] 
which is given by~\eqref{J-estimate-2} in Lemma~\ref{L.J} with~\eqref{def-P0} in Lemma~\ref{L.B-trivialeigen}, and~\eqref{PZ-1} in Lemma~\ref{L.B0-eigen} with~\eqref{PV-1}-\eqref{PV-3} in Lemma~\ref{L.B-eigen}. This proves~\eqref{T-estimates-1}.

Similarly,~\eqref{T-estimates-2} and~\eqref{T-estimates-3} are easily deduced from the explicit expression for $\P^0_n$ in~\eqref{def-P0}, and the estimates~\eqref{PV-1},\eqref{PV-2},\eqref{PV-3} as well as~\eqref{PZ-3}.
\end{proof}

We conclude this section by collecting the above information on the symbols of our operators:
\begin{Corollary}\label{C.prop}
Let $\V\in \VV$ satisfying~\eqref{C.hyp-V} and $\xi=(\xi^x,\xi^y)^\top\in\RR^d\setminus\{\z\}$, define
\begin{gather*}
 \B[\V,\xi]\eqdef \xi^x \B^x[\V]+\xi^y \B^y[\V]=\Q(\xi)^{-1}\B^x[\Q(\xi)\V]\Q(\xi)|\xi| ;\\
 \T[\V,\xi]\eqdef \Q(\xi)^{-1}\T^x[\Q(\xi)\V]\Q(\xi)  \quad \text{ and } \quad \P_{\rm f}(\xi)\eqdef \Q(\xi)^{-1} \Pifx \Q(\xi) . 
\end{gather*}
Then one has $(\T[\V,\xi])^\top=\T[\V,\xi]$, $(\T[\V,\xi]\B[\V,\xi])^\top=\T[\V,\xi]\B[\V,\xi]$ and the following estimates:
\begin{subequations}
\begin{gather}
\Norm{\B[\V,\xi]} \leq C_0\  |\xi| \quad \text{ and } \quad \Norm{(\Id-\P_{\rm f}(\xi))\B[\V,\xi]}\leq \r\  C_0\  |\xi|;
 \label{C-estimates-1}\\
 c_0 \Id \leq \T[\V,\xi]  \quad \text{ and } \quad   \Norm{\T[\V,\xi]}\leq C_0 ;
 \label{C-estimates-2}
 \end{gather}
 with $c_0^{-1}=N(1+d)$, $C_0=C(\m,h_0^{-1},\norm{\V})$; and for any $\V_1,\V_2\in\VV$ satisfying~\eqref{C.hyp-V}, one has
 \begin{align}
 \Norm{\B[\V_1,\xi]-\B[\V_2,\xi]} &\leq \r\ C_1\ \norm{\V_1-\V_2} \ |\xi|;
  \label{C-estimates-3} \\
 \Norm{\T[\V_1,\xi]\B[\V_1,\xi]-\T[\V_2,\xi]\B[\V_2,\xi]}&\leq  \r\ C_1\ \norm{\V_1-\V_2}\ |\xi|;
  \label{C-estimates-4}\\
  \Norm{\T[\V_1,\xi]-\T[\V_2,\xi]}&\leq  C_1\ \left(\norm{\Pis(\V_1-\V_2)}+\r \norm{\V_1-\V_2}\right);
 \label{C-estimates-5}\\
  \Norm{(\T[\V_1,\xi]-\T[\V_2,\xi])\P_{\rm f}(\xi)}&\leq  \r\ C_1\  \norm{\V_1-\V_2}\ .
  \label{C-estimates-6}
\end{align}
\end{subequations}
 with $C_1=C(h_0^{-1} ,\m,\norm{\V_1},\norm{\V_2})$.
\end{Corollary}

\subsection{$L^2$ energy estimate}\label{S.energy-estimates-L2}

The following proposition shows that, thanks to the structure of the system exhibited in the previous section, one is able to control the ($L^2$) energy of solutions, uniformly on a time interval independent of $\r$ small. This is the key ingredient in the proof of our main results.
\begin{Proposition}\label{P.energy-estimate-L2}
Let $\V\in  W^{1,\infty}_{t,\x}((0,T)\times \RR^d)^{N(1+d)}$ satisfying~\eqref{C.depth},\eqref{C.hyp-V} with $h_0,\nu>0$, and $\W\in C^0(\into{T};L^2(\RR^d))^{N(1+d)}$, $\R \in L^1(0,T;L^2(\RR^d))^{N(1+d)}$ be such that for any $t\in\into{T}$, one has
\begin{equation}\label{FSR-L2}
 \partial_t \W\ +\ \frac1\r\B^x[\V] \partial_x \W \ + \ \frac1\r\B^y[\V] \partial_y \W \ =\ \R .
\end{equation}
One can set $\r_0^{-1},\nu_0=C(\m,h_0^{-1},\Norm{\V}_{L^\infty((0,T)\times\RR^d)})$  such that if $\r\in(0,\r_0)$ and $\nu\geq \nu_0$, then
\begin{equation}\label{L2-estimate} \norm{\W}_{L^2}(t)\leq C_0(0) e^{C_1\bNorm{\V} t} \ \norm{\W\id{t=0}}_{L^2}+ \int_0^t e^{C_1\bNorm{\V}(t-t')} C_0(t')\norm{\R}_{L^2}(t') \ \dd t',\end{equation}
with $C_0(t)=C(\m,h_0^{-1},\norm{\V}_{W^{1,\infty}}(t))$, $C_1=C(\m,h_0^{-1},\Norm{\V}_{L^\infty((0,T)\times\RR^d)})$ and 
\[\bNorm{\V}\eqdef  \Norm{\V}_{L^\infty(0,T;W^{1,\infty})}+\r\Norm{\partial_t\V}_{L^\infty((0,T)\times\RR^d)}+\Norm{\Pis\partial_t\V}_{L^\infty((0,T)\times\RR^d)}.\]
\end{Proposition}
This section is dedicated to the proof of this result. The main ingredients are the properties of the symbol of system~\eqref{FSR-L2} as well as its symmetrizer, collected in Corollary~\ref{C.prop}. 
Energy estimates for such symmetrizable systems can be obtained thanks to Bony's paradifferential calculus associated with these symbols; see~\cite{M'etivier08}. We shall however be cautious as paradifferential calculus typically provides estimates ``up to lower order operators'': while this is sufficient for regularity aspects, this could induce order-zero but large (namely non uniformly bounded with respect to $\r$ small) remainder terms, preventing the desired uniform energy control stated in~\eqref{L2-estimate}.

This is why we decompose the symbols into a first order contribution which admits a natural quantization as a Fourier multiplier, whereas only second-order contributions will be paradifferentialized. Let us be more specific. Define
\[ \delta\B(t,\x,\xi)\eqdef \chi(\xi) \delta\B^x[\Q(\xi)\V(t,\x)] |\xi| \eqdef \chi(\xi) \big(\B^x[\Q(\xi)\V(t,x)]  - \B^x[\z] \big)|\xi|\]
where $\chi$ is a smooth non-negative cut-off function ($\chi(\xi)=0$ for $|\xi|\leq 1/2$ and $\chi(\xi)=1$ for $|\xi|\geq 1$);
and let $\mT_{i\delta\B}$ be the associated paradifferential operator (see Definition~\ref{D.paradiff}). 
Similarly, define
\[  \delta\T(t,\x,\xi) \eqdef \chi(\xi) \delta\T^x[\Q(\xi)\V(t,\x)] \eqdef \chi(\xi) \big(\T^x[\Q(\xi)\V(t,\x)]  - \T^x[\z] \big) \]
and $\mT_{\delta\T}$ the associated paradifferential operator; and
\begin{equation} \label{def-S}
  \mS(t)\eqdef  (\Q(D))^{-1}  \left( \T^x[\z] + \frac12  \left( \mT_{\delta\T} + \mT_{\delta\T}^\star\right)  +\lambda {\sf \Lambda}_\r^{-1} \right) \Q(D) ,\end{equation}
where ${\sf \Lambda}_\r^{-1}=(\Id+|D|^2)^{-1/2}(\Pisx+\r\Pifx)$ and $\lambda>0$ will be determined later on.
\medskip

We claim in Lemma~\ref{L.S}, below, that the properties on $\T[\V,\xi]$ given in Corollary~\ref{C.prop} are sufficient to show that $\mS(t)$ is a uniformly bounded and coercive operator, and show a precise estimate on its time derivative, $\mS'(t)$. In Lemma~\ref{L.paralinearization}, we then rewrite~\eqref{FSR-L2} as a symmetric, paradifferential equation, from which the energy estimate~\eqref{L2-estimate} is easily deduced.

\begin{Lemma}\label{L.S}
Let $\V$ be as in Proposition~\ref{P.energy-estimate-L2} and fix $t\in \into{T}$. Then $\mS(t):L^2\to L^2$, defined by~\eqref{def-S}, is self-adjoint and one can set $\lambda_0=C(\m,h_0^{-1},\norm{\V}_{W^{1,\infty}})$ and $\nu_0,\r_0^{-1}=C(\m,h_0^{-1},\norm{\V}_{L^\infty})$, such that if $\lambda\geq \lambda_0$ and $\nu\geq \nu_0,\r\in(0,\r_0)$, then one has 
\begin{equation}\label{S-coercive} \forall \W\in L^2(\RR^d)^{N(1+d)}, \quad  \frac{c_0}{2}\norm{\W}_{L^2}^2 \leq \big( \mS(t )\W,\W\big)_{L^2} \quad \text{ and } \quad \norm{ \mS(t )\W}_{L^2}\leq C_0\norm{\W}_{L^2} ,\end{equation}
where $c_0^{-1}=N(1+d),C_0=C(\m,h_0^{-1},\norm{\V}_{W^{1,\infty}})$. Moreover, $\mS'(t):L^2\to L^2$ is well-defined and one has
\begin{equation}\label{S-commutator} \forall \W\in L^2(\RR^d)^{N(1+d)}, \quad   \norm{\mS'(t)\W}_{L^2}\leq C_0' \big(\norm{\Pis \partial_t\V}_{L^\infty} +\r \norm{\partial_t\V}_{L^\infty}  \big)\norm{\W}_{L^2} ,\end{equation}
where $C_0'=C(\m,h_0^{-1},\norm{\V}_{L^\infty})$.
\end{Lemma}
\begin{proof}
That $\mS(t)$ is self-adjoint is obvious (recall in particular that $\Q(\xi)$ is orthogonal). Let us now decompose
\begin{align*} \Q(D)  \mS (\Q(D))^{-1} &= \T^x[\Z] + \frac12  \left( \mT_{\delta\T_{(1)}} + \mT_{\delta\T_{(1)}}^\star\right)  +\lambda{\sf \Lambda}_\r^{-1}  \\
&\hspace{2cm}+ \frac12  \left( \mT_{\delta\T_{(2)}} + \mT_{\delta\T_{(2)}}^\star\right)  - \chi(D)^{1/2} (\T^x[\Z]-\T^x[\z])\chi(D)^{1/2} \\
&\hspace{2cm}+ \chi(D)^{1/2} (\T^x[\Z]-\T^x[\z])\chi(D)^{1/2} - (\T^x[\Z]-\T^x[\z]) \\
&\eqdef \T^x[\Z] +\mS_{(1)}+\lambda{\sf \Lambda}_\r^{-1}+\mS_{(2)}+\mS_{({\rm r})} ,
\end{align*}
where $\Z$ is obtained from $\V$ by setting to zero all $n^{\rm th}$ entries with $n\geq N+1$, and  (notice  $\Q(\xi)\Z=\Z$)
\begin{align*}  
 \delta\T_{(1)}(t,\x,\xi) &\eqdef \chi(\xi)  (\T^x[\Q(\xi)\V(t,\x)] -\T^x[\Q(\xi)\Z])  , \\
 \delta\T_{(2)}(t,\x,\xi) &\eqdef \chi(\xi)  (\T^x[\Z]-\T^x[\z])  . 
  \end{align*}

Using~\eqref{C-estimates-2} in Corollary~\ref{C.prop}, we can define $C_0=C(\m,h_0^{-1},\norm{\Z}_{L^\infty})$ such that
\[\forall \W\in L^2(\RR^d)^{N(1+d)}, \quad   c_0\norm{\W}_{L^2}^2 \leq \big( \T^x[\Z] \W,\W\big)_{L^2} \quad \text{ and } \quad \norm{ \T^x[\Z] \W}_{L^2}\leq \frac12 C_0\norm{\W}_{L^2} .\]
By~\eqref{C-estimates-5} in Corollary~\ref{C.prop} and since $\V$ satisfies~\eqref{C.hyp-V}, one has for any $ (t,\x,\xi)\in \into{T}\times \RR^d\times \RR^d\setminus\{\z\}$,
\[\Norm{\T^x[\Q(\xi)\V(t,\x)] -\T^x[\Q(\xi)\Z(t,\x)]} \leq C(\m,h_0^{-1},\norm{\V(t,\x)}) (\nu^{-1}+\r\norm{\V(t,\x)}).\]
Similarly, the Lipschitz estimate~\eqref{C-estimates-5} yields for any $|\alpha|\leq 2$ and $ (t,\x,\xi)\in \into{T}\times \RR^d\times \RR^d$,
\[ \sup_{\x\in\RR^d}\Norm{\partial_\xi^\alpha \delta\T_{(1)}(t,\x,\xi) } \leq C(\m,h_0^{-1},\norm{\V}_{L^\infty},\alpha) (\nu^{-1}+\r\norm{\V}_{L^\infty}) (1+|\xi|)^{-|\alpha|},\]
so that the contribution from $\mS_{(1)} $ is estimated thanks to Proposition~\ref{P.paradiff-estimates} (item i) as follows:
\[ \Norm{\mS_{(1)}}_{L^2\to L^2}\leq  C(\m,h_0^{-1},\norm{\V}_{L^\infty}) (\nu^{-1}+\r\norm{\V}_{L^\infty}) .\]
Similarly, using~\eqref{C-estimates-6} in Corollary~\ref{C.prop} yields, uniformly in $t\in\into{T}$,
\[ \Norm{\T^x[\Z]-\T^x[\z]}_{W^{1,\infty}}+\frac1\r \Norm{(\T^x[\Z]-\T^x[\z])\Pifx}_{W^{1,\infty}}\leq  C(h_0^{-1} ,\m,\norm{\V}_{L^\infty}) \norm{\V}_{W^{1,\infty}};\]
and one obtains identically the corresponding estimates for derivatives with respect to $\xi$. Thus Propositions~\ref{P.paradiff-estimates} and~\ref{P.paradiff-particular cases} yield for any $\W\in L^2(\RR^d)^{N(1+d)}$,
\[ \norm{\mS_{(2)}\W }_{L^2}+\frac1\r \norm{\mS_{(2)} \Pifx \W }_{L^2}\leq C(\m,h_0^{-1},\norm{\V}_{L^\infty})\norm{\V}_{W^{1,\infty}} \norm{\W}_{H^{-1}}.\]
One easily checks that the last contribution satisfies the same estimate:
\[ \norm{\mS_{({\rm r})}\W }_{L^2}+\frac1\r \norm{\mS_{({\rm r})}\Pifx \W }_{L^2}\leq C(\m,h_0^{-1},\norm{\V}_{L^\infty})\norm{\V}_{W^{1,\infty}} \norm{\W}_{H^{-1}}.\]
Here, we used that for scalar functions $v\in W^{1,\infty}(\RR^d)$ and $w\in H^{-1}(\RR^d)$, one has
\begin{equation}\label{prod-H-1} \norm{vw}_{H^{-1}}\leq C(d)  \norm{v}_{W^{1,\infty}}\norm{w}_{H^{-1}},\end{equation}
where $C(d)$ is a universal constant (by duality, since for any $\varphi\in H^1$, one has $v\varphi\in H^1$ and $\langle w,v\varphi\rangle_{H^{-1}-H^1}\leq C \norm{w}_{H^{-1}}\norm{v}_{W^{1,\infty}}\norm{\varphi}_{H^1}$).

Thanks to the above estimates, one may choose $\lambda\geq \lambda_0=C(\m,h_0^{-1},\norm{\V}_{L^\infty})\norm{\V}_{W^{1,\infty}}$ sufficiently large so that one has 
\[ \forall \W\in L^2(\RR^d)^{N(1+d)}, \quad \big( (\mS_{(2)}+\mS_{({\rm r})} )\W,\W\big)_{L^2}+\lambda \big( {\sf \Lambda}_\r^{-1} (D)\W,\W\big)_{L^2} \geq 0.\]
It is now clear (using that $\Q(\xi)$ is orthogonal for $\xi\in\RR^d\setminus\{\z\}$) that one can restrict $\r,\nu,\lambda$ as in the statement so that~\eqref{S-coercive} holds.
\medskip

As for the second part of the statement, by definition of the quantization $\mT$ (see Definition~\ref{D.paradiff}) and since $\partial_t$ commutes with constant operators and Fourier multipliers, one has
\[ \Q(D) \mS'(t) (\Q(D))^{-1} =  \frac12  \left( \mT_{\partial_t (\delta\T)} + \mT_{\partial_t (\delta\T)}^\star\right).\]
Moreover, using~\eqref{C-estimates-5} in Corollary~\ref{C.prop}, one has
\[ \Norm{ \partial_t(\T^x[\V(t,\x)] )}\leq C(\m,h_0^{-1},\norm{\V(t,\x)}) \left(\norm{\Pis\partial_t\V(t,\x)}+\r\norm{\partial_t\V(t,\x)}\right).\]
Again, one easily obtains the corresponding estimates for derivatives with respect to $\xi$ and Proposition~\ref{P.paradiff-estimates} (item i) yields
\[ \Norm{\mS'(t)}_{L^2\to L^2}\leq  C(\m,h_0^{-1},\norm{\V}_{L^\infty}) \left(\norm{\Pis\partial_t\V}_{L^\infty}+\r\norm{\partial_t\V}_{L^\infty}\right).\]
Estimate~\eqref{S-commutator} is proved, and the proof of Lemma~\ref{L.S} is complete.
\end{proof}

\begin{Lemma}\label{L.paralinearization}
Let $\V,\W,\R$ be as in Proposition~\ref{P.energy-estimate-L2} and $\lambda,\r,\nu$ as in Lemma~\ref{L.S}. Then one has
\begin{equation}\label{FS-para}
 \mS \partial_t \W +   (\Q(D))^{-1} \left( i \frac1\r  \T^x[\z]  \B^x[\z] |D| +\mT_{i{\sf \Sigma}}  \right) \Q(D) \W = \mR [\V,\W,\R],
\end{equation}
where
\[  {\sf \Sigma}(t,\x,\xi) =\frac1\r \chi(\xi)\big(  \T^x[\Q(\xi)\V(t,\x)]\B^x[\Q(\xi)\V(t,\x)] - \T^x[\z]\B^x[\z]  \big)|\xi| ,\]
and $\mR[\V,\W,\R]\in L^1(0,T;L^2(\RR^d))^{N(1+d)}$ satisfies 
\[  \norm{\mR[\V,\W,\R] }_{L^2} (t)\leq C_0 \norm{\R }_{L^2}(t) +C_1 \norm{\W }_{L^2}(t)  \]
with $C_0=C(\m,h_0^{-1},\norm{\V}_{W^{1,\infty}}),C_1=C(\m,h_0^{-1},\norm{\V}_{L^\infty})\times \norm{\V}_{W^{1,\infty}}$. Moreover, one has
\[ \forall t\in \into{T}, \quad \left|\Re\big(\mT_{i{\sf \Sigma}} \W,\W\big)_{L^2} \right|\leq C(\m,h_0^{-1},\norm{\V}_{L^\infty}) \norm{\V}_{W^{1,\infty}}\norm{\W}_{L^2}^2.\]
\end{Lemma}
\begin{proof}
Notice first that, using the rotational invariance property~\eqref{B-rot}, one has
\[\B^x[\z]\partial_x\W+\B^y[\z]\partial_y \W = i (\Q(D))^{-1} \B^x[\z]  \Q(D)|D|\W. \]
Thus applying the operator $\mS$ to~\eqref{FSR-L2} yields~\eqref{FS-para}
with
\begin{multline} \label{def-mR}
 \mR[\V,\W,\R]\eqdef \mS\R \ +\ (\Q(D))^{-1}  \mT_{i{\sf \Sigma}} \Q(D) \W\\
 -(\Q(D))^{-1} \left( \frac12  \left( \mT_{\delta\T} + \mT_{\delta\T}^\star\right)  +\lambda {\sf \Lambda}_\r^{-1}(D)\right) \Q(D)\left(\frac1\r\B^x[\V]\partial_x\W+\frac1\r\B^y[\V]\partial_y \W \right) \\
 -(\Q(D))^{-1} \T^x[\z] \Q(D) \left(\frac1\r\big(\B^x[\V]-\B^x[\z]\big)\partial_x\W+\frac1\r\big(\B^y[\V]-\B^y[\z]\big)\partial_y \W \right)  .
\end{multline}
We first note that, using estimate~\eqref{C-estimates-1} in Corollary~\ref{C.prop} yields immediately
\[  \frac1\r \norm{ {\sf \Lambda}_\r^{-1}(D)\Q(D)\left(\B^x[\z]\partial_x\W+\B^y[\z]\partial_y \W \right)}_{L^2}=\frac1\r \norm{ {\sf \Lambda}_\r^{-1}(D)\B^x[\z] \Q(D) |D|\W}_{L^2}\leq C(\m,h_0^{-1},\lambda) \norm{\W}_{L^2}.\]
We then deduce from estimate~\eqref{C-estimates-3} in Corollary~\ref{C.prop} that
\[ \Norm{\B^x[\V]-\B^x[\z]}_{W^{1,\infty}}+\Norm{\B^y[\V]-\B^y[\z]}_{W^{1,\infty}}\leq  \r\ C(h_0^{-1} ,\m,\norm{\V}_{L^\infty}) \norm{\V}_{W^{1,\infty}};\]
and in turn, by~\eqref{prod-H-1}, 
\[ \forall t\in\into{T},\quad  \norm{ {\sf \Lambda}_\r^{-1}(D)\Q(D)\left(\frac1\r\B^x[\V]\partial_x\W+\frac1\r\B^y[\V]\partial_y \W \right)}_{L^2}\leq C(\m,h_0^{-1},\norm{\V}_{L^\infty}) \norm{\V}_{W^{1,\infty}} \norm{\W}_{L^2}.\]
Thus there only remains to estimate
\begin{align*}
 \mR_{(1)}[\V,\W]&\eqdef (\Q(D))^{-1} \mT_{i{\sf \Sigma}_{(1)}}\Q(D)\W- \frac12   (\Q(D))^{-1}\left( \mT_{\delta\T} + \mT_{\delta\T}^\star\right) \Q(D) \left(\frac1\r\B^x[\V]\partial_x\W+\frac1\r\B^y[\V]\partial_y \W \right) ,\\
\mR_{(2)}[\V,\W]&\eqdef (\Q(D))^{-1} \mT_{i{\sf \Sigma}_{(2)}}\Q(D)\W\\
&\hspace{2cm} -(\Q(D))^{-1} \T^x[\z] \Q(D) \left(\frac1\r \big(\B^x[\V]-\B^x[\z]\big)\partial_x\W+\frac1\r \big(\B^y[\V]-\B^y[\z]\big)\partial_y \W \right) ,
 \end{align*}
 where
\[
  {\sf \Sigma}_{(1)}(t,\x,\xi) \eqdef \frac1\r \delta\T(t,\x,\xi)\B^x[\Q(\xi)\V]  \quad \text{ and } \quad 
  {\sf \Sigma}_{(2)}(t,\x,\xi) \eqdef \frac1\r \T^x[\z]\delta\B(t,\x,\xi) .
\]
 
 These terms are estimated exactly as in the proof of Lemma~\ref{L.S}, \ie using the paradifferential calculus of Propositions~\ref{P.paradiff-estimates} and~\ref{P.paradiff-particular cases} together with the estimates of Corollary~\ref{C.prop}, thus we do not detail. Let us just indicate why these contributions are uniformly bounded with respect to $\r$ small. The case of $ \mR_{(2)}[\V,\W]$ is quickly settled by~\eqref{C-estimates-3} in Corollary~\ref{C.prop}. As for $ \mR_{(1)}[\V,\W]$, we decompose as above
 \[ \frac1\r \B^x[\V] = \frac1\r\B^x[\z] + \frac1\r \big( \B^x[\V] -\B^x[\z]\big).\]
 The contribution from the second component is uniformly bounded with respect to $\r$ small, again thanks to~\eqref{C-estimates-3} in Corollary~\ref{C.prop}. The contribution from the first component may also be uniformly bounded by remarking that
 \[ \Norm{ \delta\T(t,\x,\xi)\B^x[\z]  } \leq \Norm{\big(\delta\T(t,\x,\xi) \Pif \big) \B^x[\z] }+\Norm{\delta\T(t,\x,\xi)\big(\Pis\B^x[\z] \big) }\leq \r\times C(\m,h_0^{-1},\norm{\V})\norm{\V},\]
 where we used~\eqref{C-estimates-1},~\eqref{C-estimates-5} and~\eqref{C-estimates-6} in Corollary~\ref{C.prop}.
 
 Altogether, one estimates $\mR$ in~\eqref{def-mR} as desired, namely
\[  \norm{\mR[\V,\W,\R] }_{L^2} \leq C_0 \norm{\R }_{L^2} +C_1 \norm{\W }_{L^2}  ,\]
with $C_0=C(\m,h_0^{-1},\norm{\V}_{W^{1,\infty}}),C_1=C(\m,h_0^{-1},\norm{\V}_{L^\infty})\times \norm{\V}_{W^{1,\infty}}$.
\medskip

There remains to estimate $\Re\big(\mT_{i{\sf \Sigma}} \W,\W\big)_{L^2} $.
By~\eqref{C-estimates-4} in Corollary~\ref{C.prop}, one has
\[\forall (t,\xi)\in \into{T}\times\RR^d,\qquad  \Norm{{\sf \Sigma}(t,\cdot,\xi)  }_{W^{1,\infty}}\leq \chi(\xi)|\xi| C(\m,h_0^{-1},\norm{\V}_{L^\infty})\norm{\V}_{W^{1,\infty}} ,\]
and ${\sf \Sigma}(t,\x,\xi)$ is symmetric. Proposition~\ref{P.paradiff-estimates} (items ii. and iii.) as well as Proposition~\ref{P.paradiff-particular cases} yield
\[ \forall t\in \into{T}, \quad \left|\Re\big(\mT_{i{\sf \Sigma}} \W,\W\big)_{L^2} \right|\leq C(\m,h_0^{-1},\norm{\V}_{L^\infty}) \norm{\V}_{W^{1,\infty}}\norm{\W}_{L^2}^2.\]
The proof of Lemma~\ref{L.paralinearization} is complete.
\end{proof}

\paragraph{Completion of the proof.}

Assume $\W$ is sufficiently regular, say $\W\in C^1(\into{T};H^1(\RR^d))^{N(1+d)}$, so that all the calculations below are well-defined.
We compute the $L^2$ inner-product of identity~\eqref{FS-para} with $\W$: it follows
\begin{multline*} \frac12\frac{\dd}{\dd t}\big( \mS \W ,\W\big)_{L^2}\leq \frac12 \Big|\big( [\partial_t,\mS] \W ,\W\big)_{L^2}\Big|+\frac12 \Big|\Re\big( i \frac1\r   (\Q(D))^{-1} \T^x[\z]  \B^x[\z] \Q(D)|D| \W , \W \big)\Big|\\+\Big|\Re\big(\mT_{i{\sf \Sigma}} \W,\W\big)_{L^2} \Big|+\Big| \big(\mR [\R,\V,\W],\W\big)_{L^2}\Big|.
\end{multline*}

The second term on the right-hand-side is identically zero since $ \T^x[\z]  \B^x[\z] $ is symmetric. The other terms are estimated thanks to Cauchy-Schwarz inequality and Lemmata~\ref{L.S} and~\ref{L.paralinearization}. Altogether, this shows that provided we restrict $\r\in(0,\r_0)$ and $\nu\geq \nu_0$ as in Lemma~\ref{L.S}, one has
\begin{equation} \label{diff-inequality}
\frac12\frac{\dd}{\dd t}\big( \mS \W ,\W\big)_{L^2}\leq C_1 \norm{\W}_{L^2}^2 +  \norm{\mR [\R,\V,\W]}_{L^2}\norm{\W}_{L^2},
\end{equation}
with $C_1=C(\m,h_0^{-1},\norm{\V}_{L^\infty})\times( \norm{\V}_{W^{1,\infty}}+\r\norm{\partial_t\V}_{L^\infty}+\norm{\Pis\partial_t\V}_{L^\infty})$.
Estimate~\eqref{L2-estimate} follows from Gronwall's Lemma and the coercivity of $\mS$, \ie~\eqref{S-coercive} in Lemma~\ref{L.S}, together with the control of $\mR\in L^1(0,T;L^2)$ provided in Lemma~\ref{L.paralinearization}. The fact that the same estimate holds for general $\W\in C^0(\into{T};L^2(\RR^d))^{N(1+d)}$ solution to~\eqref{FSR-L2} may be obtained by density and thanks to a standard regularization process; see~\cite[Theorem 7.1.11]{M'etivier08}. Proposition~\ref{P.energy-estimate-L2} is proved.

\subsection{$H^s$ energy estimate}\label{S.energy-estimates-Hs}

The $L^2$ energy estimate derived in Proposition~\ref{P.energy-estimate-L2} quickly induces a similar $H^s$ estimate, for any $s>d/2+1$, by using once more the specific structure of our system of equations, namely that singular terms appear only as linear components of the system~\eqref{FS-V-compact}.
\begin{Proposition}\label{P.energy-estimate-Hs}
Let $s>d/2+1$ and $\V ,\W\in C^0(\into{T};H^s(\RR^d))^{N(1+d)} $ be such that $\V$ satisfies~\eqref{C.depth},\eqref{C.hyp-V} with $h_0,\nu>0$ and $\partial_t\V \in L^\infty((0,T)\times\RR^d)^{N(1+d)}$, and 
\[
 \partial_t \W \ + \ \frac1\r \B^x[\V] \partial_x \W \ + \ \frac1\r \B^y[\V] \partial_y \W \ = \ \R ,
\]
with $\R \in L^1(0,T;H^s(\RR^d))^{N(1+d)}$. Assume that $\r\in(0,\r_0)$ and $\nu\geq \nu_0$ with $\r_0,\nu_0$ as in Proposition~\ref{P.energy-estimate-L2}. Then one has, for any $t\in \into{T}$,
\begin{equation}\label{Hs-estimate} \norm{\W}_{H^s}(t)\leq C_0(0) e^{C_1\bNorm{\V}_s t} \ \norm{\W\id{t=0}}_{L^2}+ \int_0^t e^{C_1\bNorm{\V}_s(t-t')} C_0(t')\norm{\R}_{H^s}(t') \ \dd t',\end{equation}
with $C_0(t)=C(\m,h_0^{-1},\Norm{\V}_{L^\infty(0,t;W^{1,\infty})})$, $C_1=C(\m,h_0^{-1},\Norm{\V}_{L^\infty(0,T;H^s)})$ and
\[ \bNorm{\V}_s\eqdef  \Norm{\V}_{L^\infty(0,T; H^s)}+\r\Norm{\partial_t\V}_{L^\infty((0,T)\times\RR^d)}+\Norm{\Pis\partial_t\V}_{L^\infty((0,T)\times\RR^d)}.\]
\end{Proposition}
\begin{proof}
Denote $\W^s=\Lambda^{s}\V\in C^0(\into{T};L^2(\RR^d))^{N(1+d)} $, with $\Lambda^s=(\Id-\Delta)^{s/2}$. Then one has
\begin{equation}\label{FSV-Hs} \partial_t\W^s+\frac1\r\B^x[\V]\partial_x\W^s+\frac1\r\B^y[\V]\partial_y\W^s=\Lambda^s\R+\big[\Lambda^s,\frac1\r\B^x[\V]\big]\partial_x\W+\big[\Lambda^s,\frac1\r\B^y[\V]\big]\partial_x\W .
\end{equation}
We shall apply Proposition~\ref{P.energy-estimate-L2} to the above system, thanks to the standard tools on Sobolev spaces recalled in Section~\ref{S.product}. Notice first that since since $s>d/2+1$, Sobolev embedding yields
\[ \norm{\V}_{W^{1,\infty}}\ \leq\ C\ \norm{\V}_{H^s}, \]
so that we only need to estimate the commutator on the right-hand-side of~\eqref{FSV-Hs} to apply Proposition~\ref{P.energy-estimate-L2}. Since $\Lambda^s$ commutes with $\B^x[\z]$, one has
\[ \big[\Lambda^s,\frac1\r\B^x[\V]\big]\partial_x\W \ = \ \frac1\r\big[\Lambda^s,\B^x[\V]-\B^x[\z]\big]\partial_x\W.\]
Now, thanks to the product and commutator estimates recalled in Section~\ref{S.product}, and following the proof of Lemma~\ref{L.B}, one easily checks that, for any $t\in\into{T}$,
\[ \norm{\big[\Lambda^s,\frac1\r\B^x[\V]\big]\partial_x\W}_{L^2} =\frac1\r\norm{ \big[\Lambda^s,\B^x[\V]-\B^x[\z]\big]\partial_x\W}_{L^2}\leq C(\m,h_0^{-1} ,\norm{\V}_{H^s}) \norm{\V}_{H^s}\norm{\partial_x \W}_{H^{s-1}}.\]
Obviously, since $\B^y[\V]=\Q((0,1))^{-1}\B^x[\Q((0,1))\V]\Q((0,1))$ (by~\eqref{B-rot} ), one has
\[ \norm{\big[\Lambda^s,\frac1\r\B^y[\V]\big]\partial_y\W}_{L^2} \leq C(\m,h_0^{-1} ,\norm{\V}_{H^s}) \norm{\V}_{H^s}\norm{\partial_y \W}_{H^{s-1}}.\]

One could apply the $L^2$ estimate of Proposition~\ref{P.energy-estimate-L2} to $\V,\W^s,\Lambda^s\R$ satisfying~\eqref{FSV-Hs}, but one obtains a slightly stronger result by stepping back and using directly the differential inequality of the proof, namely~\eqref{diff-inequality}:
\[
\frac12\frac{\dd}{\dd t}\big( \mS \W^s ,\W^s\big)_{L^2}\leq C_1 \norm{\W^s}_{L^2}^2 +  \norm{\mR [\R,\V,\W]}_{L^2}\norm{\W^s}_{L^2},
\]
with $C_1=C(\m,h_0^{-1},\norm{\V}_{L^\infty})\times( \norm{\V}_{W^{1,\infty}}+\r\norm{\partial_t\V}_{L^\infty}+\norm{\Pis\partial_t\V}_{L^\infty})$, and
\[ \norm{\mR [\R,\V,\W]}_{L^2}\leq C_0 \norm{\R }_{H^s} +C_s \norm{\W }_{H^s} , \]
where $C_0=C(\m,h_0^{-1},\norm{\V}_{W^{1,\infty}})$, and $C_s=C(\m,h_0^{-1} ,\norm{\V}_{H^s}) \norm{\V}_{H^s}$.

Estimate~\eqref{Hs-estimate} follows from Gronwall's Lemma and the coercivity of $\mS$, \ie~\eqref{S-coercive} in Lemma~\ref{L.S}. Again, this estimate is proved for sufficiently regular $\W^s=\Lambda^s\W$, but may be extended to general $\W^s\in C^0(\into{T};L^2(\RR^d))^{N(1+d)}$ solution to~\eqref{FSV-Hs} by density and thanks to a standard regularization process. 
This concludes the proof of Proposition~\ref{P.energy-estimate-Hs}.
\end{proof}

\section{Well-posedness and stability estimates}\label{S.WP}
In this section, we collect the information gathered in the previous sections, which quickly yield Theorem~\ref{T.WP}, as well as the propagation of well-prepared initial data (see Proposition~\ref{P.well-prepared initial data}).

Let us first recall that Proposition~\ref{P.WP-naive} offers the existence and uniqueness of a (maximal) solution to~\eqref{FS-U-mr}, for sufficiently regular initial data. Our results give additional information on the large time behaviour of these solutions, using in particular the energy estimates of Proposition~\ref{P.energy-estimate-L2} and~\ref{P.energy-estimate-Hs}. These estimates, however, are based on a different formulation of the equations, namely~\eqref{FS-V-mr}, which we claimed to be equivalent. 
Let us precisely state below in which sense.
\begin{Lemma}\label{L.UtoV}
Let $\U\eqdef (\r^{-1}\zeta_1,\zeta_2,\dots,\zeta_N,u^x_1,\dots,u^x_N,u^y_1,\dots,u^y_N)^\top\in C^0(\into{T};H^s(\RR^d))^{N(1+d)}$ with $s>d/2+1$, satisfying~\eqref{C.depth}. Then $\V\eqdef (\r^{-1}\zeta_1,\zeta_2,\dots,\zeta_N,v^x_2,\dots,v^x_N,w^x,v^y_2,\dots,v^y_N,v^y)^\top$, defined by~\eqref{UtoV}, satisfies $\V\in C^0(\into{T};H^s(\RR^d))^{N(1+d)}$ and
\[ \norm{\V}_{H^s}\leq C(\m,h_0^{-1},\norm{\U}_{H^s}) \norm{\U}_{H^s}.\]
Conversely, if $\V\in C^0(\into{T};H^s(\RR^d))^{N(1+d)}$ satisfies~\eqref{C.depth}, then the change of variables~\eqref{VtoU} defines $\U\in C^0(\into{T};H^s(\RR^d))^{N(1+d)}$ and
\[ \norm{\U}_{H^s}\leq C(\m,h_0^{-1},\norm{\V}_{H^s}) \norm{\V}_{H^s}.\]
Moreover, if $\U$ as above is a strong solution to~\eqref{FS-U-compact} (or, equivalently,~\eqref{FS-U-mr}), then $\V$ is a strong solution to~\eqref{FS-V-compact} (or, equivalently,~\eqref{FS-V-mr}); and conversely.
\end{Lemma}
\begin{proof}
Notice that $\U\in C^0(\into{T};H^s(\RR^d))^{N(1+d)}$ implies immediately $\V\in C^0(\into{T};H^s(\RR^d))^{N(1+d)}$ since $H^\sigma(\RR^d)$ is an algebra for any $\sigma>d/2$. The converse is also true since the multiplication by $\left( \sum_{i=1}^N \gamma_i^{-1} h_i\right)^{-1}$  is continuous from $H^\sigma(\RR^d)$ to $H^\sigma(\RR^d)$; see Section~\ref{S.product}.
It follows, by continuous Sobolev embedding, that all the terms in~\eqref{FS-U-mr} and~\eqref{FS-V-mr} as well as in the calculations below are well-defined in $C^0(\into{T}\times\RR^d)^{N(1+d)}$.

The evolution equations for $\v_n$ in~\eqref{FS-V-mr} are straightforwardly deduced from the ones for $\u_n$ in~\eqref{FS-U-mr}. The evolution equation for $\w$ follows from
\[ \partial_t\w= \sum_{i=1}^N h_i  \partial_t \u_i +  \u_i \partial_t (\zeta_i-\zeta_{i+1}) =  \sum_{i=1}^N h_i  \partial_t \u_i -  \u_i  \nabla\cdot(h_i\u_i).\]
Plugging the expression for $\partial_t\u_i$ in~\eqref{FS-U-mr}, and using
$ \nabla\cdot (h_i \u_i \otimes \u_i)=h_i(\u_i\cdot \nabla)\u_i+\u_i \nabla\cdot(h_i\u_i)$, yields immediately the desired expression of the evolution equation for $\w$ in~\eqref{FS-V-mr}.

The corresponding result concerning the compact matricial formulation of the systems, and in particular~\eqref{AtoB}, is obvious by the chain rule.

This proves the first part of the statement. The second part is identical since all these calculations are reversible: the Jacobian associated to the change of variables is invertible; see~\eqref{def-J}.
\end{proof}

\subsection{Large time well-posedness; proof of Theorem~\ref{T.WP}}

By Proposition~\ref{P.WP-naive}, one can set $\nu_0=C(\m,h_0^{-1})$ such that if~\eqref{C.hyp-mr} holds with $\nu_0$, then
there exists a unique $\U\in C^0(\into{T_{\rm max}};H^s(\RR^d))^{N(1+d)}$ strong solution to~\eqref{FS-U-compact} and $\U\id{t=0}=\U^\init$. 
By Lemma~\ref{L.UtoV}, the change of variable~\eqref{UtoV} defines $\V\in C^0(\into{T_{\rm max}};H^s(\RR^d))^{N(1+d)} $ and $\V$ satisfies
\[ \partial_t\V+\frac1\r\B^x[\V]\partial_x\V+\frac1\r\B^y[\V]\partial_y\V=\z,\]
and $\norm{\V\id{t=0}}_{H^s}=C(\m,h_0^{-1},\norm{\U^\init}_{H^s})\norm{\U^\init}_{H^s}$. 
Let us denote
\[ T^\star(M)=\sup\{ t\in\into{T_{\max}} , \quad \Norm{\V}_{L^\infty(0,t;H^s)} \leq M \text{ and \eqref{C.depth-mr}-\eqref{C.hyp-mr} holds with $h_0/2,2\nu_0$}\}.\]
We restrict our discussion below to $M>\norm{\V\id{t=0}}_{H^s}$, so that (by continuity) $T^\star(M)>0$.

Using the system satisfied by $\V$,~\eqref{C-estimates-1} in Corollary~\ref{C.prop} as well as Sobolev embeddings, one checks
\[\bNorm{\V}_s\eqdef  \Norm{\V}_{L^\infty(0,T; H^s)} + \r\Norm{\partial_t\V}_{L^\infty((0,T)\times\RR^d)} +\Norm{\Pis\partial_t\V}_{L^\infty((0,T)\times\RR^d)}\leq C(\m,h_0^{-1},M)M,\]
for any $T\in\into{T^\star(M)}$; recall $\Pif$ is defined in~\eqref{def-Pi}.

In particular, one has for any $T\in\into{T^\star(M)}$, 
\[ \sup_{(t,\x)\in\intf{T}\times \RR^d}\r \norm{\partial_t(\r^{-1}\zeta_1)}+\sum_{n=2}^N \norm{\partial_t\zeta_n}+\norm{\partial_t v_n^x}+\norm{\partial_t v_n^y}\leq \bNorm{\V}_s\leq C(\m,h_0^{-1},M)M,\]
from which we deduce
\[ \forall n\in \{1,\dots,N\}, \quad \left| h_n(t,\cdot) -h_n(0,\cdot) \right| \leq  \int_0^T \left|  \zeta_{n+1}-\zeta_n\right|\leq C(\m,h_0^{-1},M)M \times T \]
and, similarly,
\[ \forall n\in \{2,\dots,N\}, \quad \norm{v_n^x}+\norm{ v_n^y} \leq  \norm{ v_n^x\id{t=0}}+\norm{ v_n^y\id{t=0}}+  C(\m,h_0^{-1},M)M \times T
.\]
It follows that, for given $M$, one can set $\nu_0,\r_0^{-1}=C(\m,h_0^{-1},M)$ such that if $\r\in(0,\r_0)$ and~\eqref{C.hyp-mr} holds with $\nu_0$, then $\V$ satisfies the assumptions of Proposition~\ref{P.energy-estimate-Hs} for any $t\in \into{\min\{T^\star(M),T^\sharp(M)\}}$ with
\[(T^\sharp(M))^{-1} =  C(\m,h_0^{-1},M)M.\]
We thus deduce the energy estimate
\[ \forall t\in\into{\min\{T^\star(M),T^\sharp(M)\}}, \quad \norm{\V}_{H^s}\leq   \big(C_0(0)\norm{\U^\init}_{H^s}+ C_M  M t \big)   e^{C_M M t} ,\]
with $C_0=C(\m,h_0^{-1},\norm{\U^\init}_{W^{1,\infty}})$ and $C_M=C(\m,h_0^{-1},M)$. In particular this shows that one can choose $M=2 C_0(0)\norm{\U^\init}_{H^s}$ such that 
\[ T^\star(M)^{-1}\leq \max\left\{T_{\max}^{-1},C(\m,h_0^{-1},\norm{\U^\init}_{H^s})\norm{\U^\init}_{H^s}\right\}.\]

Going back to the original variables through~\eqref{VtoU} and by Lemma~\ref{L.UtoV}, this shows that one can restrict $\nu_0=C(\m,h_0^{-1},\norm{\U^\init}_{H^s})$ and the time interval $\into{T}$, with $T^{-1}$ bounded as above, so that $\norm{\U}_{H^s}$ is uniformly bounded and~\eqref{C.depth-U}-\eqref{C.hyp-U} remain satisfied.
By continuity, uniqueness of the maximal solution and thanks to the blow-up criteria in Proposition~\ref{P.WP-naive}, we deduce that one can set
\[T_{\max}^{-1}\leq T^{-1}\leq C(\m,h_0^{-1},\norm{\U^\init}_{H^s})\norm{\U^\init}_{H^s}.\]
This concludes the proof of Theorem~\ref{T.WP}.

\subsection{Propagation of well-prepared initial data}
\begin{Proposition}\label{P.well-prepared initial data}
Let $s>d/2+1$ and $\U^\init\in H^{s}(\RR^d)^{N(1+d)}$ satisfying~\eqref{C.depth-mr}-\eqref{C.hyp-mr} as in Theorem~\ref{T.WP}, and denote $\U\eqdef(\r^{-1}\zeta_1,\zeta_2,\dots,\zeta_N,u_1^x,\dots,u_N^x,u_1^y,\dots,u_N^y)^\top\in C^0(\into{T};H^{s}(\RR^d))^{N(1+d)}$ the solution to~\eqref{FS-U-mr} and $\U\id{t=0}=\U^\init$. If $\U^\init$ satisfies initially
\[ \norm{\U^\init}_{H^s}+\frac1\r \norm{\nabla(\r^{-1}\zeta_1^\init)}_{L^2}+\frac1\r \norm{\sum_{n=1}^N\nabla\cdot (h_n^\init\u_n^\init)}_{L^2}\leq \ M_0,\]
then there exists $C_0= C(M_0,\m,h_0^{-1})$ such that 
\[ \forall t\in\into{T}, \quad\ \norm{\U}_{H^s}+ \frac1\r\norm{\nabla(\r^{-1}\zeta_1)}_{L^2}+\frac1\r \norm{\sum_{n=1}^N\nabla\cdot (h_n\u_n)}_{L^2}\ \leq  C_0 M_0 \exp(C_0 M_0 t),\]
uniformly with respect to $\r\in(0,\r_0)$. 
\end{Proposition}
\begin{proof}
Let us denote by $\U\in C^0(\into{T};H^{s}(\RR^d))^{N(1+d)}$ the solution to~\eqref{FS-U-mr} and $\U\id{t=0}=\U^\init$ defined by Theorem~\ref{T.WP}; and by $\V\in C^0(\into{T};H^{s}(\RR^d))^{N(1+d)} $ the associated solution to~\eqref{FS-V-compact}, namely
\[ \partial_t\V+\frac1\r\B^x[\V]\partial_x\V+\frac1\r\B^y[\V]\partial_y\V=\z,\]
obtained through the change of variables~\eqref{UtoV} (see Lemma~\ref{L.UtoV}). Finally we denote $\W =\partial_t\V$. Using the above equation and the product estimates recalled in Section~\ref{S.product}, one has immediately $\W\in C^0(\into{T};H^{s-1}(\RR^d))^{N(1+d)}\subset C^0(\into{T}\times\RR^d) $.

Differentiating the above system of equations yields
\[ \partial_t\W+\frac1\r\B^x[\V]\partial_x\W+\frac1\r\B^y[\V]\partial_y\W=  -\frac1\r  \partial_t(\B^x[\V])\partial_x\V -\frac1\r  \partial_t(\B^y[\V])\partial_y\V .\]

By construction (see the proof of Theorem~\ref{T.WP} above), one can restrict $\r_0^{-1},\nu_0=C(\m,h_0^{-1},\norm{\U^\init}_{H^s})$ so that $\V$ satisfies the assumptions of Proposition~\ref{P.energy-estimate-L2}, namely~\eqref{C.depth}-\eqref{C.hyp-V} for $t\in\into{T}$; and one has
\begin{equation}\label{est-V-WP}
\Norm{\V}_{L^\infty(0,T;W^{1,\infty})} \leq C \Norm{\V}_{L^\infty(0,T;H^{s})}\leq C(\m,h_0^{-1},M_0)M_0.
\end{equation}
By estimate~\eqref{C-estimates-1} in Corollary~\ref{C.prop}, one has
\[ \r\Norm{\partial_t\V}_{L^\infty((0,T)\times\RR^d)}+\Norm{\Pis\partial_t\V}_{L^\infty((0,T)\times\RR^d)} \leq C(\m,h_0^{-1},M_0)M_0.\]
What is more, the additional smallness assumption on $\nabla\zeta_1^\init,\nabla\cdot\w^\init$ yields
\[ \norm{\W\id{t=0}}_{L^2}\leq C(\m,h_0^{-1},M_0)M_0,\]
uniformly for $\r\in(0,\r_0)$.

Finally, by~\eqref{C-estimates-3} in Corollary~\ref{C.prop}, one has
\[\norm{\partial_t(\B^x[\V])\partial_x\V }_{L^2}+\norm{ \partial_t(\B^y[\V])\partial_y\V}_{L^2}\leq \r\ C(\m,h_0^{-1},\norm{\V}_{L^{\infty}}) \norm{\V}_{W^{1,\infty}}\norm{\partial_t \V}_{L^2}.\]

Altogether, after applying Proposition~\ref{P.energy-estimate-L2} with $\W=\partial_t\V$, one deduces
\begin{equation}\label{est-dtV-WP}
\norm{\partial_t\V}_{L^2}(t)\leq C_0 M_0 e^{C_0 M_0 t} \ \norm{\W\id{t=0}}_{L^2}+ C_0 \int_0^t e^{C_0  M_0 (t-t')}\norm{\partial_t \V}_{L^2}(t') \ \dd t',
\end{equation}
with $C_0=C(\m,h_0^{-1},M_0)$. Applying Gronwall's Lemma to $\norm{\partial_t\V}_{L^2}(t)\exp(-C_0 M_0 t) $ yields
\[ \forall t\in \into{T} , \quad \norm{\partial_t\V }_{L^2}(t) \leq C_0 M_0 \exp(C_0 M_0 t), \]
for $C_0=C(\m,h_0^{-1},M_0)$ sufficiently large.

Using again the system of equations satisfies by $\V$, namely~\eqref{FS-V-mr}, estimates~\eqref{est-V-WP} and~\eqref{est-dtV-WP} yield
\[ \frac1\r\norm{\r^{-1}\nabla\zeta_1}_{L^2}+\frac1\r\norm{\nabla\cdot\w}_{L^2}\leq C_0 M_0 \exp(C_0 M_0 t)\]
for $C_0=C(\m,h_0^{-1},M_0)$ and uniformly with $\r\in(0,\r_0)$. Proposition~\ref{P.well-prepared initial data} is proved.
\end{proof}

\section{Convergence results}\label{S.convergence}

In this section, we investigate the asymptotic behaviour of the previously obtained solutions in the limit $\r\to0$. We first show that the solutions of the free-surface system~\eqref{FS-U-mr} converge weakly towards solutions of the rigid-lid system~\eqref{RL}. Strong convergence results are then obtained, first by assuming that the initial data is well-prepared, and then by approaching the oscillatory ``defect'' through rapidly propagating acoustic waves.

\subsection{Weak convergence: the rigid-lid limit}
Our first (weak) convergence result is the following.
\begin{Proposition}\label{P.weak-convergence}
As $\r\to0$, let $\U^\init_\r\to\U^\init \in H^{s}(\RR^d)^{N(1+d)}$ satisfying~\eqref{C.depth-mr}-\eqref{C.hyp-mr}. Denote, for $\r$ sufficiently small, $\U_\r\eqdef \left(\zeta_{2,\r},\dots,\zeta_{N,\r},\u_{1,\r},\dots,\u_{N,\r}\right)^\top \in C^0(\into{T};H^s(\RR^d))^{N(1+d)-1}$ the solution to~\eqref{FS-U-mr} with $U_\r\id{t=0}=\U^\init_\r$. Then, as $\r\to 0$, $\U_\r$ converges weakly (in the sense of distributions and up to a subsequence) towards $\U^\RL\eqdef (\zeta_2,\dots,\zeta_N,\u_1,\dots,\u_N)^\top\in L^\infty(0,T;H^s(\RR^d))^{N(1+d)-1}$ a solution of the rigid-lid system~\eqref{RL}, with initial data 
\[ \U^\RL\id{t=0}=\left(\zeta_2^\init,\dots,\zeta_N^\init,\u_1^\init-\delta^{-1}\Pi\w^\init,\dots,\u_N^\init-\delta^{-1}\Pi\w^\init\right)^\top,\]
where $\delta\eqdef\sum_{n=1}^N \delta_n$, $\w^\init\eqdef\sum_{n=1}^N(\delta_n+\zeta_n^\init-\zeta_{n+1}^\init)\u_n^\init$  with convention $\zeta_1^\init=\zeta_{N+1}^\init=0$ and $\Pi\eqdef\nabla\Delta^{-1}\nabla\cdot$ is the orthogonal projection onto irrotational vector fields.
\end{Proposition}
\begin{proof}
Restricting $\r\in(0,\r_0)$ if necessary, $\U^\init_\r$ satisfies the hypotheses of Theorem~\ref{T.WP}. Thus one can define $\U_\r\eqdef (\zeta_{2,\r},\dots,\zeta_{N,\r},\u_{1,\r},\dots,\u_{N,\r})^\top \in C^0(\into{T};H^s(\RR^d))^{N(1+d)} $ from the solution to~\eqref{FS-U-mr} with initial data $\U\id{t=0}=\U^\init_\r$; and
\[ \Norm{\U_\r}_{L^\infty(0,T;H^s(\RR^d))^{N(1+d)}} \leq M,\]
with $T^{-1},M=C(\m,h_0^{-1},\norm{\U^\init}_{H^s})\norm{\U^\init}_{H^s}$.

Thus, by Banach-Alaoglu theorem on $L^\infty(0,T;H^s(\RR^d))=L^1(0,T;H^{-s}(\RR^d))'$, we can extract a weakly converging subsequence (in the sense of distributions), that we still denote $\U_\r$:
\[ \U_\r \rightharpoonup \U\eqdef (\zeta_2,\dots,\zeta_N,\u_1,\dots,\u_N)^\top,\]
with $\U\in L^\infty(0,T;H^s(\RR^d))^{N(1+d)}$. 

Let us first notice that since $\U_\r$ satisfies~\eqref{FS-U-mr} one has, uniformly for $\r\in (0,\r_0)$,
\[ \forall n\in\{2,\dots,N\}, \quad \Norm{ \partial_t \zeta_{n,\r}}_{L^\infty(0,T;H^{s-1}(\RR^d))}\leq C(\m,h_0^{-1},M) M.\]
Thus, since the embedding of $H^{s-1}(\RR^d)$ in $H^s(\RR^d)$ is locally compact and by Aubin-Lions lemma and Cantor's diagonal argument, one has (again up to the extraction of a subsequence) $ \zeta_{n,\r}  \to\zeta_n$ strongly in $C^0(\into{T};H^{s-1}_{\rm loc})$. It follows by the logarithmic convexity of Sobolev norms, that for any $s'<s$,
\[ \forall n\in\{2,\dots,N\}, \quad \zeta_{n,\r} \to \zeta_n \quad \text{ in $C^0(\into{T};H^{s'}_{\rm loc})$ } .\]
Thereafter, we fix $s'\in (1+d/2,s)$ so that $H^{s'-1}_{\rm loc}\subset C^0(\RR)$ and is an algebra. For the same reasons as above, we have also
\[ \forall n\in\{2,\dots,N\}, \quad \u_{n,\r}-\u_{n-1,\r} \to \u_{n}-\u_{n-1} \quad \text{ in $C^0(\into{T};H^{s'}_{\rm loc})^d$ } \]
and, using $(\Id-\Pi)\nabla\zeta_{1,\r}\equiv0$,
\[ \forall n\in\{1,\dots,N\}, \quad (\Id-\Pi) \u_{n,\r}\to (\Id-\Pi)\u_{n}\quad \text{ in $C^0(\into{T};H^{s'}_{\rm loc})^d$ }.\]

Let us now define
\[ \w_\r\eqdef \sum_{n=1}^N \big(\delta_n+\zeta_{n,\r} -\zeta_{n+1,\r}\big)\u_{n,\r}=\sum_{n=1}^N \delta_n\u_{n,\r} + \sum_{n=2}^N \zeta_{n,\r} (\u_{n,\r}-\u_{n-1,\r})+\r (\r^{-1}\zeta_{1,\r})\u_{1,\r}.\]
Notice now that all the (quadratic) nonlinear terms are strongly convergent, so that
\[ \w_\r\rightharpoonup \w \eqdef \sum_{n=1}^N \big(\delta_n+\zeta_{n} -\zeta_{n+1}\big)\u_{n} \]
and
\[  (\Id-\Pi) \w_\r\to (\Id-\Pi)\w\quad \text{ in $C^0(\into{T};H^{s'}_{\rm loc})^d$ } ,\]
with convention $\zeta_1=\zeta_{N+1}\equiv0$. 
Notice also that passing to the (weak) limit in the first equation of~\eqref{FS-U-mr} yields
\[ \nabla \cdot \w =0 \quad \text{ and thus }\quad  \Pi\w_\r  \rightharpoonup \z .\]

We now define
\[  \t \u_{n,\r} \eqdef  \u_{n,\r}-\delta^{-1}\Pi\w_\r \quad \text{ with } \quad \delta\eqdef\sum_{n=1}^N\delta_n .\]
Notice the identity
\[ \sum_{n=1}^N \big(\delta_n+\zeta_{n,\r} -\zeta_{n+1,\r}\big) \t \u_{n,\r} = \w_\r - \left(\sum_{n=1}^N h_{n,\r} \right) \delta^{-1}\Pi\w_\r  = (\Id-\Pi) \w_\r - \r(\r^{-1}\zeta_{1,\r}) \delta^{-1}\Pi\w_\r  \]
so that
 \[ \t\w_\r \eqdef \sum_{n=1}^N \big(\delta_n+\zeta_{n,\r} -\zeta_{n+1,\r}\big) \t \u_{n,\r}  \to (\Id-\Pi)\w \quad \text{ in $C^0(\into{T};H^{s'}_{\rm loc})$ }.\]
It follows, using the formula~\eqref{VtoU} with $\t \v_{n,\r}\eqdef \gamma_n \t\u_{n,\r}-\gamma_{n-1}\t\u_{n-1,\r}=\gamma_n (\u_{n,\r}-\u_{n-1,\r})+\r^2 r_n \t\u_{n-1,\r}$ and $\t\w_\r $, and since all the components converge strongly in $C^0(\into{T};H^{s'}_{\rm loc})$, that $\t \u_{n,\r} $ converges strongly as well. By virtue of uniqueness of the weak limit, one has
\[ \t \u_{n,\r} \to   \u_n \quad \text{ in $C^0(\into{T};H^{s'}_{\rm loc})$ }.\]
We conclude by plugging the decomposition $\u_{n,\r}=\t \u_{n,\r} +\delta^{-1}\Pi\w_\r $ into the evolution equations for velocities in~\eqref{FS-U-mr}. Notice the identity, using that $\delta^{-1}\Pi\w_\r $ is irrotational,
\[ (\u_{n,\r}\cdot\nabla)\u_{n,\r}=(\t \u_{n,\r}\cdot\nabla)\u_{n,\r}+(\u_{n,\r}\cdot\nabla)\t \u_{n,\r}-(\t\u_{n,\r}\cdot\nabla)\t \u_{n,\r}-\frac12\nabla\big(|\delta^{-1}\Pi\w_\r|^2 \big).\]
Thus the only (quadratic) term which does not involve at least one strongly convergent factor turns out to be an exact gradient and independent of $n$; and so is the unbounded (linear) component of the equation, namely $\frac1\r\nabla(\r^{-1}\zeta_{1,\r})$. This shows, passing to the weak limit all the other terms in the equation, that there exists $\nabla p$, independent of $n$, such that
\[ \forall n\in\{1,\dots,N\}, \quad \partial_t \u_n \ +\ \sum_{i=2}^n r_i \nabla\zeta_i\ + \  (\u_{n}\cdot\nabla)\u_n \ = \ -\nabla p . \]
It is straightforward to pass to the limit in the conservation of mass equations, so that
\[ \forall n\in\{2,\dots,N\}, \quad   \partial_{ t}{\zeta_n} \ + \ \sum_{i=n}^N\nabla\cdot (h_i\u_i) \ =\ 0,   \]
and we have already seen that $\nabla \cdot \w =0$. 

Thus $\U^\RL\eqdef (\zeta_2,\dots,\zeta_N,\u_1,\dots,\u_N)^\top\in L^\infty(0,T;H^s) \cap C^0(\into{T};H^{s'}_{\rm loc}) $ is a solution to~\eqref{RL}, and one checks immediately that
$ \U^\RL\id{t=0}=\left(\zeta_2^\init,\dots,\zeta_N^\init,\u_1^\init-\delta^{-1}\Pi\w^\init,\dots,\u_N^\init-\delta^{-1}\Pi\w^\init\right)^\top$.
This concludes the proof of Proposition~\ref{P.weak-convergence}.
\end{proof}

\subsection{Strong convergence for well-prepared initial data}

\begin{Proposition}\label{P.convergence-WP}
Using the notations of Proposition~\ref{P.weak-convergence}, let $\U^\init_\r\to\U^\init \in H^{s}(\RR^d)^{N(1+d)}$ as $\r\to0$, with $\U_\r^\init$ well-prepared as in Proposition~\ref{P.well-prepared initial data}. Then (up to extracting a subsequence) $\U_\r \to\U^\RL $ strongly in $C^0(\into{T};H^{s'}_{\rm loc}(\RR^d))^{N(1+d)-1}$, for all $s'<s$, as $\r\to0$. 

Moreover, $\partial_t\u_{n,\r}\rightharpoonup \partial_t\u_n\in L^\infty(0,T;L^2(\RR^d))$ and $\frac1\r \nabla(\r^{-1}\zeta_{1,\r})\rightharpoonup \nabla p^\RL\in L^\infty(0,T;L^2(\RR^d))$, where $\nabla p^\RL$ is the pressure associated with $\U^\RL $ solution to~\eqref{RL}.
\end{Proposition}
\begin{proof}
We follow the proof of Proposition~\ref{P.weak-convergence}. However, we may use additionally that, thanks to Proposition~\ref{P.well-prepared initial data} (and using the system of equations~\eqref{FS-U-mr} to control time derivatives)
\[ \Norm{ \partial_t (\r^{-1}\zeta_{1,\r}) }_{L^2}+ \sum_{n-1}^N\Norm{ \partial_t \u_{n,\r} }_{L^2} \leq C(\m,h_0^{-1},M_0) M_0.\]
It follows (up to the extraction of a subsequence)  $\u_{n,\r}  \to\u_n$ and $\Pi \w_\r\to 0$ strongly in $C^0(\into{T};L^2_{\rm loc})$, and therefore in $C^0(\into{T};H^{s'}_{\rm loc})$ for $s'<s$; and $\partial_t\u_{n,\r} \rightharpoonup \partial_t \u_n\in L^\infty(0,T;L^2(\RR^d))$ (in the sense of distributions).

By Banach-Alaoglu theorem, $\frac1\r \nabla(\r^{-1}\zeta_{1,\r}) $ has a weak limit in $ L^\infty(0,T;L^2(\RR^d))$ when $\r\to 0$, and passing to the weak limit in the velocity evolution equations shows that this limit is $\nabla p^\RL$.
\end{proof}

\subsection{Strong convergence for ill-prepared initial data; proof of Theorem~\ref{T.convergence}}\label{S.convergence-strong}

The proof of Theorem~\ref{T.convergence} is divided in three parts. We first construct a ``slow mode'' approximate solution, thanks to the appropriate rigid-lid solution. We then construct a ``fast mode'' approximate solution, satisfying an acoustic wave equation with appropriate initial data. Finally, we show that, thanks to dispersive estimates on the fast mode, the coupling effects between the two modes are small, so that the superposition of the two components solves approximately the free-surface system~\eqref{FS-U-mr} with appropriate initial data. The energy estimate of Proposition~\ref{P.energy-estimate-L2}, applied to the difference between the exact and the approximate solution, allows to conclude.
\medskip

\noindent{\em Construction of the slow mode.} Using Proposition~\ref{P.convergence-WP} with well-prepared initial data
\[\U^\init_\r\eqdef \left(0,\zeta_2^\init,\dots,\zeta_n^\init,\u_1^\init-\delta^{-1}\Pi\w^\init,\dots,\u_N^\init-\delta^{-1}\Pi\w^\init\right)^\top\]
 and artificially setting $\r\to0$, one obtains in the limit  
\[\U^\RL\eqdef (\zeta_2^\RL,\dots,\zeta_N^\RL, \u_1^\RL,\dots,\u_N^\RL)\in L^\infty(0,T;H^{s}(\RR^d))^{N(1+d)-1}\]
 a solution to~\eqref{RL} with initial data
\[ \U^\RL\id{t=0}=\left(\zeta_2^\init,\dots,\zeta_N^\init,\u_1^\init-\delta^{-1}\Pi\w^\init,\dots,\u_N^\init-\delta^{-1}\Pi\w^\init\right)^\top,\]
where we recall that $\delta\eqdef\sum_{n=1}^N \delta_n$ is the total depth, $\w^\init\eqdef\sum_{n=1}^N(\delta_n+\zeta_n^\init-\zeta_{n+1}^\init)\u_n^\init$ with convention $\zeta_1^\init=\zeta_{N+1}^\init=0$ and $\Pi\eqdef\nabla\Delta^{-1}\nabla\cdot$ the orthogonal projection onto irrotational vector fields. Moreover, the corresponding pressure satisfies $\nabla p^\RL\in L^\infty(0,T;L^2(\RR^d))$, and one has 
\begin{equation}\label{est-URL} 
\norm{\U^\RL}_{L^\infty(0,T;H^{s}(\RR^d))}\leq C(\m,h_0^{-1},\norm{\U^\init}_{H^s})\norm{\U^\init}_{H^s}
\end{equation}
(since $\U_\r$, the solutions of~\eqref{FS-U-mr} with $\U_\r\id{t=0}=\U^\init$ from which $\U^\RL$ is constructed, satisfy the same estimate by Theorem~\ref{T.WP}).

Let us prove that one may (uniquely) choose $p^\RL\in L^\infty(0,T;H^s(\RR^d))$. The regularity of $\U^\RL$ strong solution to~\eqref{RL} is sufficient to claim that $\nabla p^\RL\in L^\infty(0,T;L^2(\RR^d))$ is uniquely defined by
\[ \left( \sum_{n=1}^N \delta_n\right) \Delta p^\RL \ + \ \sum_{n=1}^N \nabla\cdot \left(   h_n^\RL(\u_n^\RL\cdot\nabla)\u_n^\RL+\u_n^\RL \nabla\cdot(h_n^\RL\u_n^\RL) + h_n^\RL \sum_{i=2}^n r_i\nabla\zeta_i^\RL \right) \ = \ 0,\]
 where $h_n^\RL=\delta_n+\zeta_n^\RL-\zeta_{n+1}^\RL$ with convention $\zeta_1^\RL=\zeta_{n+1}^\RL\equiv 0$. Notice now that\[  h_n^\RL(\u_n^\RL\cdot\nabla)\u_n^\RL+\u_n^\RL \nabla\cdot(h_n^\RL\u_n^\RL) =\nabla\cdot (h_n^\RL \u_n^\RL \otimes \u_n^\RL)\]
and
 \[ \sum_{n=1}^N h_n^\RL \sum_{i=2}^n r_i\nabla\zeta_i^\RL = \sum_{i=2}^N r_i\nabla\zeta_i^\RL  \sum_{n=i}^N h_n^\RL = \sum_{i=2}^N r_i\nabla\zeta_i^\RL  \left(\zeta_i^\RL+\sum_{n=i}^N \delta_n\right) .\]
 It follows, using that $H^{s}(\RR^d)$ is an algebra, that there exists $\varphi^{x,x},\varphi^{x,y},\varphi^{y,y}\in L^\infty(0,T;H^{s}(\RR^d))$ such that
 \[ \Delta p^\RL = \partial_x^2 \varphi^{x,x}+\partial_x\partial_y  \varphi^{x,y}+\partial_y^2 \varphi^{y,y}.\]
 Thus one may define a unique solution $p^\RL\in L^\infty(0,T;H^{s}(\RR^d))$ by Fourier analysis.
 \footnote{Let us remark incenditally that such a choice of pressure with bounded energy is allowed thanks to the Boussinesq approximation applied to the rigid-lid system. Without the Boussinesq approximation, generic initial conditions will generate horizontal pressure imbalances, which in turn yield an apparently paradoxical evolution in time of the total horizontal momentum; see~\cite{CamassaChenFalquiEtAl13}. One can check that the total horizontal momentum is preserved for the free-surface system as well as for the rigid-lid system with Boussinesq approximation.
  }
 From the product estimates in Sobolev spaces (see Section~\ref{S.product}) and estimate~\eqref{est-URL}, one has
 \[ \norm{p^\RL}_{L^\infty(0,T;H^s(\RR^d))}\leq C(\m,h_0^{-1},\norm{\U^\init}_{H^s})\norm{\U^\init}_{H^s}.\]
 
We denote $\U^{\rm slow}\eqdef(\r p^\RL,\zeta_2^\RL,\dots,\zeta_N^\RL, \u_1^\RL,\dots,\u_N^\RL)\in L^\infty(0,T;H^{s}(\RR^d))^{N(1+d)}$. If $\r$ is chosen sufficiently small, then $\U^{\rm slow}$ satisfies~\eqref{C.depth-mr} and therefore the change of variables~\eqref{UtoV} defines $\V^{\rm slow}$ satisfying (see Lemma~\ref{L.UtoV})
\begin{equation}\label{est-Vslow}  \Norm{\V^{\rm slow}}_{L^\infty(0,T;H^s(\RR^d))}
\leq C(\m,h_0^{-1},\norm{\U^\init}_{H^s})\norm{\U^\init}_{H^s}.
\end{equation}

It is easy to check, using $\gamma_i=1-\r^2 \sum_{j=i+1}^N r_j$ and since $\nabla\cdot \big(\sum_{n=1}^N h_n^\RL \u_n^\RL\big)=0$, that $\V^{\rm slow}$ satisfies~\eqref{FS-V-mr} up to a remainder term denoted $\R^{\rm slow}\in L^\infty(0,T;H^{s-1}(\RR^d))$; and
\begin{equation}\label{est-Rslow} 
 \norm{\R^{\rm slow}}_{L^\infty(0,T;H^{s-1}(\RR^d))} \leq \r\  C(\m,h_0^{-1},\norm{\U^\init}_{H^s})\norm{\U^\init}_{H^s}.
 \end{equation}
\medskip

\noindent{\em Construction of the fast mode.} We constructed above an approximate solution of~\eqref{FS-U-mr} (in the sense of consistency), but which does not fulfil the required initial condition. We correct this defect through an explicit ``fast mode''. Define $\V^{\rm fast}\eqdef (\r^{-1}\zeta_1^{\rm ac},0,\dots,0,w^{x,\rm ac},0,\dots,0,w^{y,\rm ac})^\top$ where $\r^{-1}\zeta_1^{\rm ac}$ and $\w^{\rm ac}\eqdef (w^{x,\rm ac},w^{y,\rm ac})^\top$ solve
\begin{equation}\label{acoustic}
\left\{ \begin{array}{l}
\displaystyle\partial_{ t}( \r^{-1}\zeta_1^{\rm ac}) \ + \ \frac1\r \nabla\cdot \w^{\rm ac} \ =\ 0,  \\ 
\\
\dsp\partial_t \w^{\rm ac}+ \frac{\delta}{\r}\nabla(\r^{-1}\zeta_1^{\rm ac})\ =\ \z,
\end{array}
\right. 
\end{equation}
with initial condition $\zeta_1^{\rm ac}\id{t=0}=\zeta_1^\init$ and $\w^{\rm ac}\id{t=0}=\Pi \w^\init$, so that
  \begin{equation}\label{est-initial data} 
\norm{\V^{\rm fast}\id{t=0}+\V^{\rm slow}\id{t=0}-\V^\init}_{L^2} \leq \r^2\ C(\m,h_0^{-1},\norm{\U^\init}_{H^s})\norm{\U^\init}_{H^s} ,
\end{equation}
where $\V^\init$ is defined from $\U^\init$ after the change of variables~\eqref{UtoV}.

The above is an acoustic wave equation and is well understood. The following results may be found in~\cite{BahouriCheminDanchin11} for instance. There exists a unique solution $(\r^{-1}\zeta_1^{\rm ac},\w^{\rm ac})\in C^0(\RR;H^{s}(\RR^d))^{1+d}$. It satisfies $\Pi \w^{\rm ac}=\w^{\rm ac}$ for any $t\in\RR$, and
\begin{align}\label{est-Vfast} 
\forall t\in\RR, \quad \left(\norm{\r^{-1}\zeta_1^{\rm ac}}_{H^s}^2+\frac1{\delta}\norm{\w^{\rm ac}}_{H^s}^2\right)^{1/2} &= \left(\norm{\zeta_1^\init}_{H^s}^2+ \frac1\delta\norm{\w^\init}_{H^s}^2\right)^{1/2}\\
&\leq C(\m,h_0^{-1},\norm{\U^\init}_{H^s})\norm{\U^\init}_{H^s}.\nonumber
 \end{align}
What is more, since $d=2$, one has the Strichartz estimates
\[ \Norm{\r^{-1}\zeta_1^{\rm ac}}_{L^p_t(\RR;L^q(\RR^d))}+\Norm{\w^{\rm ac}}_{L^p_t(\RR;L^q(\RR^d))} \leq C \r^{1/p}\left( \norm{\zeta_1^\init}_{H^{\sigma}}+ C \norm{\w^\init}_{H^{\sigma}}\right)\]
where $p,q$ are admissible, namely $2<p,q<\infty$ and $\frac1p+\frac{d}{q}=\frac{d}2-\sigma$ and $\frac2p+\frac{d-1}{q}=\frac{d-1}2$. Set for instance $p=q=6$ and $\sigma=\frac12$. By a scaling argument and differentiating once the system, one has
  \begin{equation}\label{est-Strichartz} 
  \Norm{\V^{\rm fast}}_{L^6(0,T;W^{1,6}(\RR^d))} \leq  \r^{1/6} \ C(\m,h_0^{-1},\norm{\U^\init}_{H^s})\norm{\U^\init}_{H^s} .
  \end{equation}
It follows that we control quadratic nonlinearities as
  \begin{align*} \Norm{\V^{\rm fast}\otimes\V^{\rm fast} }_{L^1(0,T;H^1(\RR^d))}  &\leq \Norm{\V^{\rm fast} }_{L^{6}(0,T;W^{1,6}(\RR^d))} \times  \Norm{\V^{\rm fast}}_{L^{6/5}(0,T;W^{1,3}(\RR^d))} \\
  &\leq  \r^{1/6} \ C(\m,h_0^{-1},\norm{\U^\init}_{H^s}),
 \end{align*}
where we used Hölder inequality, then Sobolev embedding and~\eqref{est-Vfast}-\eqref{est-Strichartz}, with the restriction on the time interval,
\[T^{-1}\leq C(\m,h_0^{-1},\norm{\U^\init}_{H^s})\norm{\U^\init}_{H^s}.\]
 
We deduce that $\V^{\rm fast}$ satisfies the equations~\eqref{FS-V-mr} up to a remainder $\R$ such that
\begin{equation}\label{est-Rfast}  
\norm{\R^{\rm fast}}_{L^1(0,T;L^2(\RR^d))} \to 0 \quad (\r\to 0).
\end{equation}
\medskip

\noindent {\em Control of the coupling terms.} Let us denote $\V^{\rm app}\eqdef \V^{\rm slow}+\V^{\rm fast}$. 
Let us first check that $\V^{\rm app}$ is an approximate solution to~\eqref{FS-V-mr}, in the sense of consistency. From the above, we have that
\[ \partial_t \V^{\rm app} + \frac1\r\B^x[\V^{\rm app}  ] \partial_x\V^{\rm app} + \frac1\r\B^y[\V^{\rm app}  ] \partial_y\V^{\rm app} =\R^{\rm slow}+\R^{\rm fast}+\R^{\rm coupl}\]
where $\R^{\rm slow}$ and $\R^{\rm fast}$ have been defined and estimated previously; and
\begin{multline*}
 \R^{\rm coupl}\eqdef \frac1\r \big(\B^x[\V^{\rm app}  ] - \B^x[\V^{\rm fast} ] \big)\partial_x\V^{\rm fast} + \frac1\r \big(\B^y[\V^{\rm app}  ] - \B^y[\V^{\rm fast} ] \big)\partial_y\V^{\rm fast} \\
 +\frac1\r \big(\B^x[\V^{\rm app}  ] - \B^x[\V^{\rm slow} ] \big)\partial_x\V^{\rm slow} + \frac1\r \big(\B^y[\V^{\rm app}  ] - \B^y[\V^{\rm slow} ] \big)\partial_y\V^{\rm slow} .
 \end{multline*}
By estimate~\eqref{C-estimates-3} in Corollary~\ref{C.prop}, one has
\[  \norm{\R^{\rm coupl}}_{L^2}\leq C(\m,h_0^{-1},\norm{\V^{\rm fast}}_{L^{\infty}},\norm{\V^{\rm slow}}_{L^{\infty}}) \norm{\V^{\rm fast}\otimes \V^{\rm slow}}_{H^1}.\]
We deduce from~\eqref{est-Vslow},~\eqref{est-Vfast} and~\eqref{est-Strichartz}, proceeding as for~\eqref{est-Rfast},
\begin{equation}\label{est-Rcoupl} 
 \norm{\R^{\rm coupl}}_{L^1(0,T;L^2(\RR^d))}\to 0 \qquad (\r\to 0).
 \end{equation} 

Let us now denote $\U\in C^0(\into{T};H^{s}(\RR^d))^{N(1+d)}$ the strong solution to~\eqref{FS-U-mr} with initial data $\U\id{t=0}=\U^\init$ as defined by Theorem~\ref{T.WP}; and $\V$ the corresponding solution to~\eqref{FS-V-mr} defined by the change of variables~\eqref{UtoV}. By Theorem~\ref{T.WP} and Lemma~\ref{L.UtoV}, one has
\begin{equation}\label{est-V} 
 \norm{\V}_{L^\infty(0,T;H^{s}(\RR^d))} \leq C(\m,h_0^{-1},\norm{\U^\init}_{H^{s}})\norm{\U^\init}_{H^{s}}.
 \end{equation}
Proceeding as above, we find that the difference between the exact and the approximate solution, $\W\eqdef \V^{\rm app}-\V$, satisfies
\[ \partial_t\W + \frac1\r\B^x[\V ] \partial_x\W+ \frac1\r\B^y[\V ] \partial_y \W=\R^{\rm slow}+\R^{\rm fast}+\R^{\rm coupl}+\R,\]
with
\[
 \R\eqdef \frac1\r \big(\B^x[\V^{\rm app}  ] - \B^x[\V ] \big)\partial_x\V^{\rm app} + \frac1\r \big(\B^y[\V^{\rm app}  ] - \B^y[\V] \big)\partial_y\V^{\rm app} .\]
Again, by estimate~\eqref{C-estimates-3} in Corollary~\ref{C.prop}, one has
\begin{equation}\label{est-R}   \norm{\R}_{L^2}\leq C(\m,h_0^{-1},\norm{\V^{\rm app}}_{W^{1,\infty}},\norm{\V}_{L^{\infty}}) \norm{\V^{\rm app}-\V}_{L^2}.
\end{equation}
\medskip

We now apply Proposition~\ref{P.energy-estimate-L2} with $\W=\V^{\rm app}-\V$, and deduce
\[ \norm{\V^{\rm app}-\V}_{L^2}(t)\leq \norm{\V^{\rm app}-\V}_{L^2}(0) e^{C_1\bNorm{\V}t} + \int_0^t e^{C_1\bNorm{\V}(t-t')} C_0(t')\norm{\R^{\rm slow}+\R^{\rm fast}+\R^{\rm coupl}+\R}_{L^2}(t') \ \dd t',\]
with $C_0(t)=C(\m,h_0^{-1},\norm{\V}_{W^{1,\infty}}(t))$, $C_1=C(\m,h_0^{-1},\Norm{\V}_{L^\infty(0,T\times\RR^d)})$.

Using the above control on the initial data~\eqref{est-initial data} and remainder terms~\eqref{est-Rslow},\eqref{est-Rfast},\eqref{est-Rcoupl},\eqref{est-R} and Gronwall's Lemma, we obtain
\[ \sup_{t\in \into{T}} \norm{\V^{\rm app}-\V}_{L^2}(t)\to 0 \qquad (\r\to 0).\]
By the logarithmic convexity of Sobolev norms and estimates~\eqref{est-Vslow},\eqref{est-Vfast},\eqref{est-V}, it follows
\[ \sup_{t\in \into{T}} \norm{\V^{\rm app}-\V}_{H^{s'}}(t)\to 0 \qquad (\r\to 0)\]
for any $s'<s$. 
\medskip

This concludes the proof of Theorem~\ref{T.convergence}, with the exception of the uniqueness of the strong solution of the rigid-lid system~\eqref{RL}. However, given two solutions 
\[\U_j^\RL\eqdef(p^\RL_j,\zeta_{2,j}^\RL,\dots,\zeta_{N,j}^\RL, \u_{1,j}^\RL,\dots,\u_{N,j}^\RL)^\top\in L^\infty(0,T;H^{s}(\RR^d))^{N(1+d)}\qquad  (j=1,2)\]
 with same initial data, we may construct as above (\ie multiplying the pressure with a $\r$ prefactor) $\U_{1,\r}^{\rm slow},\U_{2,\r}^{\rm slow}\in L^\infty(0,T;H^{s}(\RR^d))^{N(1+d)}\cap W^{1,\infty}_{t,\x}((0,T)\times\RR^d)$, two families of approximate solutions of the free-surface system~\eqref{FS-V-compact}, for arbitrarily small $\r$. Applying Proposition~\ref{P.energy-estimate-L2} with the equation satisfied by the difference between the two solutions (after the change of variable~\eqref{UtoV}), using~\eqref{est-Vslow},~\eqref{est-Rslow} and taking the limit $\r\to0$, yields $\U^\RL_1=\U^\RL_2$.

\appendix 

\section{Notations, functional setting and technical tools}\label{S.notations}

\subsection{Notations}

We denote by $C(\lambda_1, \lambda_2,\dots)$ a non-negative constant depending on the parameters
 $\lambda_1$, $\lambda_2$,\dots and whose dependence on $\lambda_j$ is always assumed to be non-decreasing. It may also depend without acknowledgment on the horizontal dimension, $d$; number of layers, $N$; Sobolev index at stake, $s$. 
 
   Given $X$ a topological vector space, $X'$ denotes its continuous dual, endowed with the strong topology. 
 
 Given $H$ a Hilbert space and $\mT:H\to H$ a continuous linear operator, we denote $\mT^\star$ its adjoint.
 \medskip
 
For $1\leq p<\infty$, we denote $L^p=L^p(\RR^d)$ the standard Lebesgue spaces associated with the norm
 \[\vert f \vert_{L^p}\eqdef \left(\int_{\RR^d} \abs{ f(\x)}^p\  \dd \x\right)^{\frac1p}<\infty.\] 
 
  The space $L^\infty=L^\infty(\RR^d)$ consists of all essentially bounded, Lebesgue-measurable functions
   $f$ endowed with the norm
  \[
  \norm{f}_{L^\infty}\eqdef {\rm ess\,sup}_{\x\in\RR^d} | f(\x)|<\infty.
  \]

For $k\in\NN$, we denote by $W^{k,\infty}=W^{k,\infty}(\RR^d)=\{f \mbox{ s.t. } \forall 0\leq|\alpha|\leq k,\ \partial^\alpha f \in L^{\infty}(\RR^d)\}$ (endowed with its canonical norm) where we use the standard multi-index notation for $\alpha$-differentiation.

We denote by $C^k=C^k(\RR^d)$ the space of continuous functions on $\RR^d$ with continuous derivatives up to the order $k$, endowed with the same norm.

The real inner product of any functions $f_1$ and $f_2$ in the Hilbert space $L^2(\RR^d)$ is denoted by
 \[
  \big(f_1\ ,\ f_2\big)_{L^2}\eqdef\int_{\RR^d}f_1(\x)\overline{f_2(\x)}\ \dd \x.
  \]
  
  For any real constant $s\in\RR$, $H^s=H^s(\RR^d)$ denotes the Sobolev space of all tempered distributions, $f\in\mathcal S'(\RR^d)$, such that $\vert f\vert_{H^s}=\norm{ \Lambda^s f}_{L^2} < \infty$, where $\Lambda$ is the Fourier multiplier \[\Lambda\eqdef (\Id-\Delta)^{1/2}= (\Id+|D|^2)^{1/2}.\]
   \medskip

For any function $u=u(t,\x)$ defined on $ [0,T)\times \RR^d$ with $T>0$, and any of the previously defined functional spaces, $X$, we denote $L^\infty(0,T;X)$ the space of functions such that $u(t,\cdot)$ is controlled in $X$, uniformly for $t\in[0,T)$, and use double bar symbol for the associated norm:
  \[\Norm{u}_{L^\infty(0,T;X)} \ = \ {\rm ess\,sup}_{t\in[0,T)}\norm{u(t,\cdot)}_{X} \ < \ \infty.\]
  
 For $k\in\NN$, $C^k([0,T);X)$ denotes the space of $X$-valued continuous functions on $[0,T)$ with continuous derivatives up to the order $k$. 
Finally, we denote, in order to avoid confusions,
\[  W^{1,\infty}_{t,\x}=W^{1,\infty}_{t,\x}((0,T)\times\RR^d)= \{f, \mbox{ such that } f,\ \nabla f , \ \partial_t f \in L^{\infty}((0,T)\times\RR^d)\}.\]
\medskip

We denote by $\big(\cdot,\cdot\big)$ and $\norm{\cdot}$ the Euclidean inner product and norm on vector space $\RR^N$ or $\CC^N$ and $\Norm{\cdot}$ the corresponding induced norm on $\M_N(\RR)$, the space of $N$-by-$N$ matrices with real entries (the choice of the norms has little significance).

If the entries belong to a Banach algebra $X$ (\eg $X=W^{k,\infty}(\RR^d)$ or $X=H^s(\RR^d)$ with $s>d/2$), then we denote $\big(\cdot,\cdot\big)_X$, $\norm{\cdot}_X$ and $\Norm{\cdot}_X$ the corresponding inner product, vector and matrix norms.

\subsection{Product and commutator estimates}\label{S.product}
We quickly recall the standard product, Schauder and Kato-Ponce estimates in $ H^\sigma(\RR^d)$, $\sigma>d/2$. Proofs or references concerning the following results may be found for instance in~\cite{Lannes}. 

Consider scalar functions $f,g\in H^\sigma(\RR^d)$. By Sobolev embedding, one has $f,g\in C^0(\RR^d)$ and
\[ \norm{f}_{L^\infty}\leq C \norm{f}_{H^\sigma}.\]
It follows that the product is well-defined and continuous. Moreover, one has $fg\in H^\sigma(\RR^d)$ and
\[ \norm{f g}_{H^\sigma} \leq C \norm{f}_{H^\sigma}\norm{g}_{H^\sigma}.\]
Moreover, if $F\in C^k(\RR)$ with $k\in\NN$, $k\geq \sigma$ and satisfies $F(0)=0$, then $F(f)\in H^\sigma(\RR^d)$ and
\[ \norm{F(f)}_{H^\sigma} \leq C(\norm{f}_{H^\sigma},\norm{F}_{C^k})\norm{f}_{H^\sigma}.\]
A consequence of the above is of particular importance in our setting. Let $g\in H^\sigma(\RR^d)$ be such that $1+g\geq h_0>0$. Then for any $f\in H^\sigma(\RR^d)$, $\frac{f}{1+g}\in H^\sigma(\RR^d)$ and
\[\norm{\frac{f}{1+g}}_{H^\sigma} \leq C(h_0^{-1},\norm{g}_{H^\sigma})\norm{f}_{H^\sigma}\]
(it suffices to remark $\frac{f}{1+g}=f-f \frac{g}{1+g}$, and apply the Schauder estimate with $F$ a smooth function such that $F(X)=\frac{X}{1+X}$ if $|X|\leq 1-h_0$ and $F(X)=0$ if $|X|\geq 1-h_0/2$). Finally, we have the celebrated Kato-Ponce estimate for commutators:
\[ \norm{ \big[\Lambda^\sigma, f\big] \partial_x g}_{L^2} \leq C \norm{f}_{H^{\sigma}}\norm{\partial_x g}_{H^{\sigma-1}}\leq C \norm{f}_{H^{\sigma}}\norm{g}_{H^\sigma}.\]

\subsection{Paradifferential calculus}\label{S.para}

In this section, we recall some results concerning Bony's paradifferential calculus. We follow the definition and most of the notations of~\cite{M'etivier08}, although the latter is restricted to scalar functions and operators. The generalizations to (finite dimensional) vector spaces brings no additional difficulty since each operation may be reduced to a linear combination of entrywise scalar operations.

\begin{Definition}[Symbols]\label{D.symbols}
Let $k\in\NN$ and $m\in\NN$. We denote $\Gamma^m_k$ the space of locally bounded functions $A(\x,\xi):\RR^d\times \RR^d\to \M_{N(1+d)}(\CC)$ which are $C^\infty$ with respect to $\xi$ and such that for any $\alpha\in\NN^d$, $\x\mapsto \partial_\xi^\alpha A(\x,\xi)$ belongs to $W^{k,\infty}(\RR^d)$ and there exists a constant $C_\alpha$ such that
\[ \forall\xi\in \RR^d, \qquad \Norm{\partial_\xi^\alpha A(\cdot,\xi)}_{W^{k,\infty}}\leq C_\alpha (1+|\xi|)^{m-|\alpha|},\]
where we use the standard multi-index notation for $\alpha$-differentiation. For $A\in \Gamma^m_k$, we denote
\[ M^m_k(A;n)\eqdef \sup_{|\alpha|\leq n}\sup_{\xi\in\RR^d} \Norm{ (1+|\xi|)^{|\alpha|-m} \partial_\xi^\alpha A(\x,\xi) }_{W^{k,\infty}}.\]
\end{Definition}

Given a symbol $A\in \Gamma^m_k$, one can associate a suitable paradifferential operator.

\begin{Definition}[Paradifferential operators]\label{D.paradiff}
Let $A\in \Gamma^m_k$. Then we define the paradifferential operator $\mT_A$, with symbol $A$, by
\begin{equation}\label{def-para}
\forall U\in L^2(\RR^d)^{N(1+d)}, \qquad \mT_AU(\x)\eqdef (2\pi)^{-d}\int_{\RR^d} e^{i\x\cdot\xi }\Sigma_A^\psi(\x,\xi)\widehat{U}(\xi)\ \dd\xi,
\end{equation}
where $\Sigma_A^\psi$ is defined with $\mF_\x \Sigma_A^\psi(\eta,\xi)=\psi(\eta,\xi) \mF_\x A(\eta,\xi)$ where $\mF_\x A(\eta,\xi)$ is the (component-by-component) Fourier transform of $A(\x,\xi)$ with respect to $\x$, and $\psi$ is an {\em admissible cut-off} in the sense of \cite[Definition 5.1.4]{M'etivier08}, and whose expression does not need to be precised.
\end{Definition}

We now give some results used in this work.
\begin{Proposition}[\cite{M'etivier08}, (5.25) and Theorems 6.1.4,~6.2.4] \label{P.paradiff-estimates}
Let $m\in \NN$ and assume $d\leq 2$.
\begin{enumerate}
\item If $A\in \Gamma^m_0$, then for any $s\in\RR$, $\mT_A:H^s\to H^{s-m}$ is bounded, and
\[ \Norm{\mT_A}_{H^s\to H^{s-m}} \leq C(s) M_0^m(A;2).\]
\item If $A\in \Gamma^m_1$ and $B\in \Gamma^{m'}_1$, then $\mT_A\mT_B-\mT_{AB}:H^s\to H^{s-(m+m'-1)}$ is bounded, and
\[ \Norm{\mT_A\mT_B-\mT_{AB}}_{H^s\to H^{s-(m+m'-1)}} \leq C(s) M_1^m(A;3)M_1^m(B;3).\]
\item If $A\in \Gamma^m_1$, then $(\mT_A)^\star- \mT_{\overline{A}}:H^s\to H^{s-(m-1)}$ is bounded, and
\[ \Norm{(\mT_A)^\star- \mT_{\overline{A}}}_{H^s\to H^{s-(m-1)}} \leq C(s) M_1^m(A;3).\]
\end{enumerate}

\end{Proposition}
We will also make use of the particular cases of Fourier multipliers and paraproducts:
\begin{Proposition}[\cite{M'etivier08}, Theorems 5.1.15,5.2.8]\label{P.paradiff-particular cases} Let $A\in \Gamma^m_k$ with $m,k\in\NN$
\begin{enumerate}
\item If $A$ depends only on $\xi $: $A=A(\xi)\in C^\infty(\RR^d)$, then $\mT_A= A(D)$
where $A(D)$ is the Fourier multiplier associated with $A$, \ie
\[ \forall U\in L^2(\RR^d)^{N(1+d)},\ \forall\xi\in\RR^d,\qquad \widehat{\mT_AU}(\xi)  = \widehat{A(D)U}(\xi) = A(\xi)\widehat{U}(\xi).\]
\item If $A$ depends only on $\x$: $A=A(\x)\in W^{k,\infty}(\RR^d)$ with $k\in\NN$, then
\begin{align*}
\Norm{\mT_A U- A U}_{H^k} &\leq C(k) \Norm{  A }_{W^{k,\infty} } \norm{U}_{L^2},\\
\forall |\alpha|\leq k, \quad  \Norm{\mT_A \partial^\alpha U- A \partial^\alpha U}_{L^2} &\leq C(k) \Norm{  A }_{W^{k,\infty} } \norm{U}_{L^2}.\end{align*}
\end{enumerate}
\end{Proposition}

\section{Eigenstructure of our system}\label{S.spectral}

In this section, we give some information on the eigenstructure of the operators at stake in the multilayer shallow water model~\eqref{FS-V-mr}, namely $\B^x[\V]$ defined in~\eqref{FS-V-compact}; the translation in terms of the initial formulation~\eqref{FS-U-intro}, or $\A^x[\U]$, is immediate through the similarity transform~\eqref{AtoB}. More precisely, we show that the above matrices are semisimple provided that each layer's depth is positive and the shear velocities are sufficiently small. This provides the complete eigenstructure of the full symbol $\xi^x\B^x[\V]+\xi^y\B^y[\V]$ thanks to the rotational invariance property (see Lemma~\ref{L.B}). We use this result in order to construct a symmetrizer of the system with the desired properties described in Section~\ref{S.normal-form} (Lemma~\ref{L.T}), but we believe that such information is of independent interest.

Indeed, despite numerous works on the subject, the available information on the domain of hyperbolicity and eigenstructure of the multilayer shallow-water model is very sparse outside of the one or two-layer situation. The one-layer case is very classical and there is no need to discuss the subject here. Pioneer work on the two-layer case, in the limit $\r\ll 1$ and dimension $d=1$ setting, include~\cite{SchijfSchonfled53} for the free-surface case and~\cite{Long56} for the rigid-lid situation. We let the reader refer to~\cite{BarrosChoi08} and references therein for the case of $\r\geq \r_0>0$. Additional, numerical information may be found in~\cite{Castro-DiazFernandez-NietoGonzalez-VidaEtAl11}. Sufficient criteria for the hyperbolicity of the bi-fluidic shallow water model in the general situation of dimension $d=2$ and free-surface situation are provided in~\cite{Duchene10,Monjarret}, while the rigid-lid setting is treated in~\cite{GuyenneLannesSaut10,BreschRenardy11}. Starting from $N= 3$ layers, explicit results become out of reach except for very specific situations; see~\cite{ChumakovaMenzaqueMilewskiEtAl09,Frings12,StewartDellar13}. 

In the general case of $N$ layers, the author provides in~\cite{Duchene} a non-explicit sufficient condition for strict hyperbolicity in the case of dimension $d=1$, and for $\r\geq\r_0>0$, using that the eigenproblem, in absence of shear velocities, reduces to a finite-difference analogue of the Sturm-Liouville problem (a similar tridiagonal reduction already appeared much earlier in the literature; see~\cite{Baines88} and references therein). In~\cite{Monjarreta}, Monjarret provides a sufficient criterion for hyperbolicity, for $d=1$ or $d=2$, by exhibiting the symmetrizer recalled in Section~\ref{S.WP-naive}. Very precise information concerning the eigenstructure are also provided in an asymptotic limit which does not fit in our situation when $N\geq 3$, since it requires a sharp scale separation for densities between each layer; see~\cite[(4.1)]{Monjarreta}.  In this section, we build upon these works, by showing that the strategy of~\cite{Duchene} extends to dimension $d=2$ and arbitrarily small density contrast $\r$; under hypotheses on the flow which are fully consistent (although, again, non-explicit) with the criteria given in~\cite{Monjarreta}. 

Roughly speaking, we show that provided that shear velocities are not too large, the nonlinear evolution of the flow maintains $N$ modes of propagation. One of them is the barotropic mode, and is responsible for the singular time oscillations of the system in the limit $\r\to0$. In our scaling, the wave speed of the $N-1$ baroclinic modes are uniformly bounded from above and from below, and remain isolated if the shear velocities are sufficiently small.
\medskip

Let us fix $h_0>0$ and denote
\begin{align*} 
\VV&=\big\{  (\r^{-1}\zeta_1,\zeta_2,\dots,\zeta_N,\star,\dots,\star)^\top\in \RR^{N(1+d)} \text{ such that \eqref{C.depth} holds with $h_0$} \big\}, \\
\ZZ&=\big\{  (\r^{-1}\zeta_1,\zeta_2,\dots,\zeta_N,0,\dots,0)^\top\in \RR^{N(1+d)} \text{ such that \eqref{C.depth} holds with $h_0$} \big\}.
\end{align*}

We first remark that in the case of dimension $d=2$, there are $N$ trivial eigenvalues of $\frac1\r\B^x[\V]$.
\begin{Lemma}\label{L.B-trivialeigen}
Let $\V\in \VV$. There are $N$ linearly independent eigenvectors of $\frac1\r\B^x[\V]$, with corresponding eigenvalue $\mu^0_n\eqdef u^x_n(\V)$, as given by~\eqref{VtoU}, associated with rank-one eigenprojections
\begin{equation} \label{def-P0}
\P^0_n[\V] \eqdef (\J^F[\V])^{-1} {\sf \Pi}_{2N+n}   \J^F[\V]  ,
\end{equation}
where $ {\sf \Pi}_n$ is the orthogonal projection onto the $n^{\rm th}$ variable. 
\end{Lemma}
\begin{proof}
The result is straightforward when using the block formulation of $ \A^x[\U] $,~\eqref{def-A}, and the similarity transformation~\eqref{AtoB}, \ie $\B^x[\V] = (\J^F[\V])^{-1}\A^x[F(\V)] \J^F[\V]$.
\end{proof}

In the following Lemmata, we give sufficient conditions on $\V\in \VV$ allowing to complete the basis of eigenvectors with distinct and real associated eigenvalues.
\begin{Lemma}\label{L.B0-eigen}
Let $\Z\in \ZZ$. Then for any $\r>0$, $\frac1\r\B^x[\Z] $ has $2N$ distinct, real, non-zero eigenvalues, $\mu_{\pm n}(\Z)=\pm \mu_n(\Z)$, $n\in \{1,\dots,N\}$, with
\[ -\mu_1(\Z)<\dots<-\mu_N(\Z)<0<\mu_N(\Z)<\dots<\mu_1(\Z).\]
Moreover, one can set $\r_0^{-1},C_0=C(\m,h_0^{-1},\norm{\Z})$ such that if $\r\in(0,\r_0)$, then
\begin{subequations}
\begin{gather}  
\label{muZ-1}
C_0^{-1}\leq \r|\mu_{1}(\Z)|\leq C_0 \quad \text{ and } \quad C_0^{-1}\leq|\mu_n(\Z)|\leq C_0,\\
\label{muZ-2}|\mu_n(\Z)-\mu_{n-1}(\Z)|\geq C_0^{-1}\quad (n\geq 2).
  \end{gather}
The associated eigenprojections satisfy (recall the definition of $\Pifx$ in~\eqref{def-Pi})
\begin{align} 
\label{PZ-1} \Norm{ \P_{\pm 1}[\Z]  }+ \Norm{ \P_{\pm n}[\Z] } &\leq   C(\m,h_0^{-1},\norm{\Z}) \qquad (n\geq 2),\\
\label{PZ-2}
\Norm{ \P_{\pm 1}[\Z] \Pisx }+ \Norm{ \Pisx \P_{\pm 1}[\Z]  }& \leq \r \  C(\m,h_0^{-1},\norm{\Z}),\\
\label{PZ-3}
\Norm{ \P_{\pm n}[\Z] \Pifx }+ \Norm{ \Pifx \P_{\pm n}[\Z] }&\leq \r \  C(\m,h_0^{-1},\norm{\Z})\ \qquad (n\geq 2) .
\end{align}
All these objects are smooth with respect to $\r\in(0,\r_0)$ and $\Z\in \ZZ$. 
\end{subequations}
\end{Lemma}
\begin{proof}
The eigenvalue problem concerning $\frac1\r\B^x[\Z] $ is related by~\eqref{AtoB} to the one for $\frac1\r\A^x[\Z]$, which given the simple block-structure exhibited in~\eqref{def-A} may be reduced, for $\Z\in\ZZ$, to the eigenvalue problem for the following tridiagonal matrix
\[ ({\sf H}{\sf R})^{-1}=\D(\e_\r)\D(\t r)^{-1}{\sf \Delta}^{\top}\D(h^{-1})\D(\gamma){\sf \Delta}\D(\e_\r)\]
 (recall notations and identities in the proof of Proposition~\ref{P.WP-naive}). Equivalently, we consider the eigenvalue problem for
\[ \T^\r\eqdef \D(\t r)^{1/2} ({\sf H}{\sf R})^{-1}\D(\t r)^{-1/2}=\D(\e_\r)\D(\t r)^{-1/2}{\sf \Delta}^{\top}\D(h^{-1})\D(\gamma){\sf \Delta}\D(\t r)^{-1/2}\D(\e_\r).\]
Since $\T^\r$ is a real, symmetric tridiagonal matrix with non-zero (positive) off-diagonals entries~\cite{Wilkinson65}, there exists $\lambda_1<\dots<\lambda_N$ and $(\x_1,\dots,\x_N)$ an orthonormal basis of $\RR^N$ such that
\[ \forall n\in \{1,\dots, N\}, \quad \T^\r \x_n = \lambda_n \x_n.\]
Since $\gamma_n,r_n,h_n^{-1},\r>0$, one may check that $\det_n(\T^\r)>0$ where $\det_n(\T^\r)$ is the determinant of the $n$-by-$n$ upper-left submatrix (\ie leading principal minor) of $\T^\r$, from which we deduce $\lambda_1>0$.

Now, let us consider
\[ \T^0=\D(\e_0)\D(\t r)^{-1/2}{\sf \Delta}^{\top}\D(h^{-1})\D(\gamma){\sf \Delta}\D(\t r)^{-1/2}\D(\e_0),\]
\ie the matrix obtained from $\T^\r$ by setting to zero the first row's and column's entries. Using the above analysis on the leading principal minor of order 1, one obtains immediately that there exists
$ 0=\lambda_1^0<\lambda_2^0<\dots< \lambda_N^0$ and $\x_n^0$, such that $ \T^0 \x_n^0 = \lambda_n \x_n^0$ ($n=1,\dots,N$). The above eigenvalues and eigenvectors depend continuously on the parameters which are, by assumption, bounded in a compact set; in particular there exists $C_0=C(\m,h_0^{-1},\norm{\Z})$, for a given $N$, such that 
\[ C_0^{-1} <\lambda_2^0<\dots< \lambda_N^0<C_0 \quad \text{ and } \quad  |\lambda_n^0-\lambda_{n-1}^0|>C_0^{-1}\qquad (n=2,\dots,N).\]
Again by continuity, one can set $\r_0^{-1}=C(\m,h_0^{-1},\norm{\Z})$ such that if $\r\in(0,\r_0)$, then $\lambda_n$ satisfy the above (replacing $C_0$ with $2C_0$). Furthermore, we have (augmenting $C_0$ if necessary)
\[  \r^2 \ C_0^{-1} \leq \lambda_1 = \det(\T^\r) \times \left(\Pi_{n=2}^N \lambda_n \right)^{-1} \leq \r^2 C_0.\]

Going back to the original problem, one deduces from~\eqref{AtoB} and~\eqref{def-A} that for any $n\in\NN$, $\mu_{\pm n}\eqdef \pm \lambda_n^{-1/2}$ is an eigenvalue of $\frac1\r\B^x[\Z] $.
This concludes the proof of the first part of the statement with~\eqref{muZ-1},\eqref{muZ-2}.

Notice now that the above analysis is not restricted to $\r>0$. Indeed, the formula for $\T^\r$ above defines a (complex) tridiagonal matrix for any $\r\in\CC$ and is single-valued, as least for $\r$ sufficiently small. Since the off-diagonal entries do not vanish, $\T^\r$, and therefore $\frac1\r\B^x[\Z]$ has $2N$ distinct and non-zero eigenvalues for any $\r\neq 0$. Moreover, we are in the situation of~\cite[Theorem II.1.10]{Kato95}, namely $\T^\r$ is normal for a sequence $\r_i\to 0$ (simply restricting to $\r_i\in\RR$), and therefore $\lambda_n(\r)$ and the associated eigenprojection are holomorphic around $\r=0$.

 The extra information concerning the corresponding eigenprojections in the limit of vanishing $\r$ are obtained thanks to standard perturbation theory~\cite[Chapter II.1]{Kato95}. Indeed, from~\eqref{FS-V-mr} (see also the proof of Lemma~\ref{L.B}), one can write for any $\r\in\CC$,
\[ \B^x[\Z] = \L+ \r\ \delta\B[\Z] \]
with
\[ \L=\left(\begin{MAT}(b){c:c:c}
{\sf 0}_N& {\sf N} & {\sf 0}_N \\:
 \alpha{\sf N}^\top &{\sf 0}_N &  {\sf 0}_N\\:
  {\sf 0}_N& {\sf 0}_N& {\sf 0}_N \\
\end{MAT}\right), \quad {\sf N}=\left(\begin{MAT}{ccc:c}
0 &\dots  & 0 & 1 \\:
&&& 0 \\
 &{\sf 0}_{N-1}&& \vdots \\
  &&& 0\\
\end{MAT}\right) , \quad \alpha=\gamma_1\sum_{j=1}^N\delta_j;\]
and $\delta\B[\Z]$ is smooth with respect to $\Z$ and holomorphic with respect to $\r$, and satisfies
\[ \Norm{\delta \B[\Z]}\leq C(\m,h_0^{-1},\norm{\Z})  .\]
It is obvious that $\L$ has only two non-zero eigenvalues:
\[ \L \x^\L_\pm = \pm \sqrt{\alpha} \x^\L_\pm,\]
with $\Pisx \x^\L_\pm=\z$.

Using that $\r\mu_{\pm 1}$ (resp. $\pm\sqrt{\alpha}$) is a simple eigenvalue of $\B^x[\Z]$ (resp. $\L$),  we introduce the Dunford-Taylor integral for the eigenprojection 
\[
\P_{\pm1}[\Z]=\frac{-1}{2\pi i}\int_{\Gamma_{\pm1}}\sum_{k=0}^\infty (\L-\eta\Id)^{-1}(-\r \delta\B[\Z] (\L-\eta\Id)^{-1})^k\ d\eta,
\]
where $\Gamma_{\pm1}$ is a positively oriented closed curve enclosing $\r\mu_{\pm 1}$ as well as $\pm\sqrt{\alpha}$ , but excluding the other eigenvalues of $\B^x[\Z]$ and $\L$, namely $0$ and $\r\mu_{\pm n}$ ($n\geq 2$). One can restrict $|\r|<\r_0$, $\r_0^{-1}=C(\m,h_0^{-1},\norm{\Z})$ such that $|\r|\Norm{ \delta\B[\Z]} \max_{\eta\in \Gamma_{\pm}}\Norm{ (\L-\eta\Id)^{-1}}\leq 1/2$, and therefore the series in the Dunford-Taylor integral is uniformly convergent. In particular,
\[ \P_{\pm1}[\Z] =\frac{-1}{2\pi i}\int_{\Gamma_{\pm1}} (\L-\eta\Id)^{-1} \ d\eta+\frac{\r }{2\pi i}\int_{\Gamma_{\pm1}}(\L-\eta\Id)^{-1}  \sum_{k=1}^\infty \r^{k-1}\left(- \delta\B [\Z] (\L-\eta\Id)^{-1}\right)^k\ d\eta.\]
The first term on the right hand side is exactly the eigenprojection onto $\x^\L_\pm$, and the series in the second term is immediately estimated. We deduce 
\[ \Norm{ \P_{\pm 1}[\Z]   } \leq  C(h_0^{-1},\r,\norm{\Z}), \quad \Norm{ \Pisx \P_{\pm 1}[\Z]  } + \Norm{ \P_{\pm 1}[\Z] \Pisx  } \leq  \r\  C(h_0^{-1},\r,\norm{\Z}),\]
thus the first part of~\eqref{PZ-1} and~\eqref{PZ-2} is proved.
\medskip

One cannot directly use the same technique for the other eigenvalues, as they correspond to an {\em exceptional} point of $\B^x[\Z]=\L+\r\delta\B[\Z]$, namely $0$ is an eigenvalue of $\L$ with algebraic multiplicity $N(1+d)-2$, and splits  for $\r\neq 0$ in $2N-2$ distinct eigenvalues, $\r \mu_{\pm n}(\Z),n\geq2$ (completed with the $(d-1)N$ linearly independent eigenvectors given in Lemma~\ref{L.B-trivialeigen} which remain in the kernel). As a consequence, the Dunford-Taylor integral around this group of eigenvalues only yields a control of the total projection, namely the sum of the corresponding eigenprojections; or, if one integrates around a single eigenvalue, one cannot in general ensure that the series is convergent, even for $\r$ small.

 However, we have seen that $\lambda_n(\r)$ is holomorphic in $\r$ and converges towards $\lambda_n(0)>0$ as $\r\to0$, so that $\r\mu_{\pm n}(\Z)=\pm \r \lambda_n(\r)^{-1/2}$ is holomorphic near $\r=0$. This means~\cite[Theorem~II.1.8]{Kato95} that the exceptional point is not a {\em branch} point, and therefore $\P_{\pm n}[\Z]$ is single-valued, but still may have a pole at $\r=0$. Using now that $0$ is a semisimple eigenvalue of the unperturbed operator, $\L$, and that we have shown that $\r \mu_{\pm n}(\Z)$, the eigenvalues of $\B^x[\Z]$ splitting from $0$ are simple (as they are one-dimensional), we may use the so-called reduction process, and deduce~\cite[Theorem~II.2.3]{Kato95} that the associated spectral projections, $\P_{\pm n}[\Z]$, are actually holomorphic at $\r=0$. 
 
 This shows~\eqref{PZ-1}, and by continuity that $\P_{\pm n}[\Z]\id{\r=0}$ is a projection onto the kernel of $\L$. Thus $(\P_{\pm n}[\Z]\id{\r=0})\Pifx =\Pifx(\P_{\pm n}[\Z]\id{\r=0})={\sf 0}_{N(1+d)}$, and~\eqref{PZ-3} follows. Lemma~\ref{L.B0-eigen} is proved.
\end{proof}

We now deduce from Lemma~\ref{L.B0-eigen}, thanks to standard perturbation theory, the corresponding information on the eigenvalue problem for any $\V\in\VV$.
\begin{Lemma}\label{L.B-eigen}
Let $\V\in \VV$. Then one can set $\r_0^{-1},\nu=C(\m,h_0^{-1},\norm{\V})$ such that if $\r\in(0,\r_0)$ and $\V$ satisfies additionally
\[\forall n\in \{2,\dots, N\}, \quad  \big| v_n^x\big|+\big| v_n^y\big| \leq \nu^{-1},\]
then the matrix $\frac1\r\B^x[\V]$ is diagonalizable. In addition to the $N$ ``trivial'' eigenvectors described in Lemma~\ref{L.B-trivialeigen}, $\frac1\r\B^x[\V]$ has $2N$ eigenvectors corresponding to distinct and real eigenvalues, $\mu_{\pm n}(\V)$, such that
\[ \mu_{-1}(\V)<\dots<\mu_{-N}(\V)<\mu_N(\V)<\dots<\mu_1(\V).\]
Moreover, there exists $C_0=C(\m,h_0^{-1},\norm{\V})$ such that 
\begin{subequations}
\begin{gather}  
\label{muV-1}
C_0^{-1}\leq \r|\mu_{\pm 1}(\V)|\leq C_0 \quad \text{ and } \quad |\mu_{\pm n}(\V)|\leq C_0 \qquad (n\geq 2),\\
\label{muV-2}|\mu_{- N}(\V)-\mu_{N}(\V)|\geq C_0^{-1} \quad \text{ and } \quad |\mu_{\pm n}(\V)-\mu_{\pm(n-1)}(\V)|\geq C_0^{-1} \qquad (n\geq 2).
  \end{gather}
The associated spectral projections are smooth with respect to $\V\in \VV$; and, for any $\V_1,\V_2\in \VV$,
\begin{align}
\label{PV-1}
\Norm{ \P_{\pm 1}[\V_1] -  \P_{\pm 1}[\V_2]  } &\leq C_1\ \r \norm{\V_1-\V_2},\\
 \label{PV-2}
\Norm{  (\P_{\pm n}[\V_1]  - \P_{\pm n}[\V_2])\Pifx } &\leq C_1\ \r \norm{\V_1-\V_2} \qquad (n\geq 2),\\
\label{PV-3}
 \Norm{ \P_{\pm n}[\V_1] - \P_{\pm n}[\V_2] } &\leq \ C_1 \big(\norm{\Pis(\V_1-\V_2)}+\r \norm{\V_1-\V_2}\big) \qquad (n\geq 2),
\end{align}
where $C_1=C(\m,h_0^{-1},\norm{\V_1},\norm{\V_2})$.
\end{subequations}
\end{Lemma}
\begin{proof}
We shall use perturbation arguments, starting from the knowledge that, thanks to Lemma~\ref{L.B0-eigen}, the non-zero eigenvalues of $\B^x[\Z]$, for any $\Z\in \ZZ$, are simple. In what follows, we denote $\Z\in \ZZ$ the vector obtained from setting to zero all $n^{\rm th}$ entries of $\V$ with $n\geq N+1$. Identifying~\eqref{FS-V-mr} with~\eqref{FS-V-compact}, we may write (improving the description provided in Lemma~\ref{L.B})
\begin{equation}\label{B-decomp}
\B^x[\V] = \B^x[\Z] +\r K^x  \Id + \r \delta\B_{\rm f}[\V]+ \r \delta\B_{\rm s}[\V],
\end{equation} 
where one can choose $K^x=K^x(\V)$ (for example $K^x=(\sum_{i=1}^N\delta_i)^{-1}w^x$) such that $\delta\B_{\rm s}$ contains (at first order in $\r$) only contributions from shear velocities $v_2,\dots,v_N$, while $\delta\B_{\rm f}$ contains a leading-order contribution in $w^x,w^y$, but only on ``fast variable'' entries.
 More precisely, one has
\begin{align}\label{est-deltaB}
\Norm{ \delta\B_{\rm s}[\V] } + \Norm{ \delta\B_{\rm f}[\V] } &\leq \ C(\m,h_0^{-1},\norm{\V}) \norm{\V} , \\
\label{est-deltaBs}
 \Norm{ \delta\B_{\rm s}[\V]}  &\leq \ C(\m,h_0^{-1},\norm{\V}) (\nu^{-1}+\r\norm{\V} ),   \\
 \label{est-deltaBf}
\Norm{  \delta\B_{\rm f}[\V] \Pisx } &\leq \ \r\  C(\m,h_0^{-1},\norm{\V}) \norm{\V} .
 \end{align}

 Since all the non-trivial eigenvalues of $\B^x[\Z]$ (and therefore the ones of $\B[\Z]+\r K^x\Id$) are simple, we may use the Dunford-Taylor integral
\begin{equation}\label{eqn:Cauchy-B}
\P_{\pm n}[\V]=\frac{-1}{2\pi i}\int_{\Gamma_{\pm n}}\sum_{k=0}^\infty \R(\eta)\big(-\r(\delta\B_{\rm s}+\delta\B_{\rm f}) \R(\eta)\big)^k\ d\eta,
\end{equation}
where $\Gamma_{\pm n}$ is a positively oriented closed curve enclosing the eigenvalue $\r\mu_{\pm n}(\Z)+\r K^x$, but excluding the other eigenvalues of $\B[\Z]+\r K^x\Id$, and
\begin{equation} \label{R-decomp}
\R(\eta) \eqdef  (\B^x[\Z]+\r K^x\Id-\eta\Id)^{-1}= \sum_{n'=1}^N \frac{ \P_{n'}[\Z] }{ \r\mu_{n'}(\Z)+\r K^x-\eta}+\frac{ \P_{-n'}[\Z]}{\r \mu_{-n'}(\Z)+\r K^x-\eta}+\frac{\P^0_{n'}[\Z]}{\r K^x-\eta},
\end{equation}
with $\mu_{\pm n}(\Z),\P_{\pm n}[\Z],\P^0_n[\Z]$ have been defined in Lemmata~\ref{L.B-trivialeigen} and~\ref{L.B0-eigen}.
\medskip

Let us first consider the case $n=1$. From Lemma~\ref{L.B0-eigen}, there exists $C_0=C(\m,h_0^{-1},\norm{\Z})$ such that if $n=1$, one can choose $\Gamma_{\pm 1}$ as the circle of center $\r\mu_{\pm1}(\Z)$ and of radius $C_0^{-1}$. Using~\eqref{est-deltaB} and~\eqref{R-decomp}, one can restrict $\r\in(0,\r_0)$ with $\r_0^{-1}=C(\m,h_0^{-1},\norm{\V})$ such that the series in~\eqref{eqn:Cauchy-B} is immediately convergent and
\[ \Norm{\P_{\pm 1}[\V]-\P_{\pm 1}[\Z ] } \leq \frac{1}{\pi } \sum_{n=1}^\infty \r^n \big(\sup_{\Gamma_{\pm n}}\Norm{ \R(\eta)}\big)^{1+n} \big(\Norm{ \delta\B_{\rm s}}+\Norm{\delta\B_{\rm f}}\big)^n\leq \r\  C(\m,h_0^{-1},\norm{\V}) \norm{\V},
\]
since $\P_{\pm 1}[\Z ] =\frac{-1}{2\pi i}\int_{\Gamma_{\pm n}}\R(\eta)\ d\eta$ is the first term of the series. Since $ \P_{\pm 1}[\Z],\P_{\pm 1}[\V]$ are rank-one, one has
 \[\r \mu_{\pm 1}(\V) \ = \ \tr\big( \B^x[\V] \P_{\pm1}[\V]\big)\ = \ \r \mu_{\pm 1}(\Z) \ + \ \tr\big( \B^x[\V]  ( \P_{\pm1}[\V]- \P_{\pm1}[\Z])\big)\ + \ \tr\big( (\B^x[\V] -\B^x[\Z]) \P_{\pm1}[\Z]\big) ,\]
 and the upper and lower bound on $\r|\mu_{\pm1}(\V)|$ in~\eqref{muV-1} follows from the ones on $\r|\mu_{\pm1}(\Z)|$ given in~\eqref{muZ-1}, and the above estimate with~\eqref{B-estimates} in Lemma~\ref{L.B}.

We now turn to the case $2\leq n\leq N$. In this case, we set $\Gamma_{\pm n}$ as the circle of center $\r\mu_{\pm1}(\Z)$ and of radius $\r C_0^{-1}$ (with $C_0$ as in~\eqref{muZ-2}), thus we may only ensure
\[ \sup_{\Gamma_{\pm n}}\Norm{ \R(\eta)} \leq \r^{-1}C(\m,h_0^{-1},\norm{\Z} ).\]
As a consequence, we need the precised estimates~\eqref{est-deltaBs}-\eqref{est-deltaBf} as well as the decomposition into partial fraction~\eqref{R-decomp} to ensure that the series converge. Indeed, thanks to~\eqref{est-deltaBs}, one may augment $\nu$ and lower $\r_0$ in order to ensure
\[ \r \sup_{\eta\in\Gamma_{\pm n}}\Norm{ \R(\eta)} \Norm{ \delta\B_{\rm s}[\V]}\leq  C(\m,h_0^{-1},\norm{\V})   (\r\norm{\V}+\nu^{-1}) \leq \frac12.\]
Now, we use that for any $n'\geq 2$,
\[ \delta\B_{\rm f}[\V]  \P_{\pm n'}[\Z] =\big(   \delta\B_{\rm f}[\V] \Pisx \big)  \P_{\pm n'}[\Z] +\delta\B_{\rm f} [\V] \big( \Pifx \P_{\pm n'}[\Z] \big) .\]
Thus by~\eqref{PZ-3} and~\eqref{est-deltaBf}, one has
\[  \Norm{ \delta\B_{\rm f} [\V]\P_{\pm n'}[\Z]  } \leq  C(\m,h_0^{-1},\norm{\V}) \ \r\norm{\V} \qquad (n'\geq 2).\]
The contribution from $\P^0_{n'}[\Z]$ is straightforwardly estimated, and the contribution from $\P_{\pm n'}$ with $n'=1$ contains no difficulty, since, $( \r\mu_n(\Z)+\r K^x-\eta)^{-1}$ is uniformly bounded. Altogether, one has
\[ \sup_{\eta\in\Gamma_{\pm n}} \Norm{ \r \delta\B_{\rm f} [\V]\R(\eta )} \leq  C(\m,h_0^{-1},\norm{\V}) \ \r\norm{\V}.\]
Restricting $\r\in(0,\r_0)$ if necessary,  the series in~\eqref{eqn:Cauchy-B} converges and
\[ \Norm{\P_{\pm n}[\V]-\P_{\pm n}[\Z ] } \leq \  C(\m,h_0^{-1},\norm{\V})   (\r\norm{\V}+\nu^{-1}) .
\]
In the same way, and using~\eqref{PZ-3}, one easily sees that
\[ \Norm{ \r \R(\eta ) \Pifx } \leq  C(\m,h_0^{-1},\norm{\V}) \ \r\norm{\V},\]
and therefore
\[ \Norm{ (\P_{\pm n}[\V]-\P_{\pm n}[\Z ]) \Pifx} \leq \  \r\  C(\m,h_0^{-1},\norm{\V})  \norm{\V} .
\]

In particular, provided $\r\in(0,\r_0)$ and $\nu^{-1}$ are sufficiently small, $\Gamma_{\pm n}$ contains exactly one eigenvalue of $\frac1\r\B^x[\V]$,
 and~\eqref{muV-1}-\eqref{muV-2} follow.

Estimates~\eqref{PV-1},\eqref{PV-2},\eqref{PV-3} are obtained identically as above, using the decomposition 
\[ \B^x[\V_2]=\B^x[\V_1]+\r (K^x_2 -K^x_1) \Id+\delta\B[\V_1,\V_2]\]
where
\[ \delta\B[\V_1,\V_2] \ \eqdef \  \B^x[\Z_2]-\B^x[\Z_1] + \r (\delta\B_{\rm f}[\V_2]-\delta\B_{\rm f}[\V_1])+ \r (\delta\B_{\rm s}[\V_2]-\delta\B_{\rm s}[\V_1]) ,
 \]
and remarking that
\begin{align*}
|K^x_2 -K^x_1|+\frac1\r \Norm{ \B^x[\Z_2]-\B^x[\Z_1]} +\Norm{\delta\B_{\rm s}[\V_2]-\delta\B_{\rm s}[\V_1]} & \leq C_1 (\norm{\Pis(\V_1-\V_2)}+\r\norm{\V_1-\V_2} ) \ ,\\
 \Norm{\delta\B_{\rm f}[\V_2]-\delta\B_{\rm f}[\V_1] }+\frac1\r \Norm{ (\delta\B_{\rm f}[\V_2]-\delta\B_{\rm f}[\V_1]) \Pisx}& \leq C_1 \norm{\V_1-\V_2}\ ,
\end{align*}
with $C_1=C(\m,h_0^{-1},\norm{\V_1},\norm{\V_2})$. This concludes the proof of Lemma~\ref{L.B-eigen}.
\end{proof}

\begin{Remark}
The proof of Lemma~\ref{L.B-eigen} is somewhat cumbersome and rely on delicate properties of $\B^x[\V]$, namely~\eqref{B-decomp} with~\eqref{est-deltaBs} and~\eqref{est-deltaBf}, because we wanted to be as precise as possible as for the hyperbolicity conditions (see Remark~\ref{R.naive}). The proof is considerably shortened and appears more robust if one replaces the assumption of Lemma~\ref{L.B-eigen} with the more stringent $\norm{\V}\leq \nu^{-1}$, as one may then simply use $\B^x[\V]=\B^x[\Z] +\r\delta\B[\V]$ with $\norm{\delta\B[\V]}\leq C(\m,h_0^{-1},\norm{\V})\norm{\V}$ in lieu of~\eqref{B-decomp}.
\end{Remark}

\paragraph{Acknowledgements.} 
The author is partially supported by the Agence Nationale de la Recherche, project ANR-13-BS01-0003-01 DYFICOLTI.

\end{document}